\RequirePackage{ifthen}
\RequirePackage{ifpdf}
\newcommand{\driverOption}{}
\ifthenelse{\boolean{pdf}}{
  \renewcommand{\driverOption}{pdftex}
} { % else
  \renewcommand{\driverOption}{dvips}
}

\documentclass[reqno, 11pt, letterpaper, commented, oneside, \driverOption]{amsart}

\newboolean{isCommented}
\usepackage{geometry}
\usepackage{bbm}
\usepackage[final]{graphicx}
\usepackage{upref}
\usepackage{enumerate}
\usepackage{latexsym}
\usepackage{amssymb}
\usepackage[ansinew]{inputenc}
\usepackage[T1]{fontenc}
\usepackage{mathtools}
\usepackage{mycommands}

\ifthenelse{\boolean{pdf}}{
  \usepackage[final]{ps4pdf}
} { % else
  \usepackage[inactive]{ps4pdf}
}
\PSforPDF{\usepackage{psfrag}}

\newcommand{\hyperrefDriverOption}{}
\ifthenelse{\boolean{pdf}}{
	\renewcommand{\hyperrefDriverOption}{pdftex}
} { % else
	\renewcommand{\hyperrefDriverOption}{hypertex}
}
\usepackage[\hyperrefDriverOption,
  colorlinks = false, 
  pdfauthor={Torsten Mütze, Franziska Weber},
  pdftitle={Construction of 2-factors in the middle layer of the discrete cube}]
  {hyperref}
    
\ifthenelse{\boolean{isCommented}} {
	\newcommand{\TM}[1]{\marginpar{\parbox{4cm}{{\small {\bf TM:} #1}}}}
	\newcommand{\FW}[1]{\marginpar{\parbox{4cm}{{\small {\bf FW:} #1}}}}
} { % else
	\newcommand{\TM}[1]{}
	\newcommand{\FW}[1]{}
}
\newtheorem{theorem}{Theorem}
\newtheorem{lemma}[theorem]{Lemma}

\newtheorem{proposition}[theorem]{Proposition}

\theoremstyle{definition}

\theoremstyle{remark}
\newtheorem{remark}[theorem]{Remark}

\long\def\symbolfootnote[#1]#2{\begingroup
\def\thefootnote{\fnsymbol{footnote}}\footnote[#1]{#2}\endgroup}

\ifthenelse{\boolean{isCommented}} {
	\geometry{%showframe,
	  hmargin={25mm, 50mm},
	  marginparwidth=40mm, 
	  vmargin={25mm, 25mm},
	  headsep=10mm,
	  headheight=5mm,
	  footskip=10mm
	}
} { % else
	\geometry{%showframe,
	  hmargin={25mm, 25mm}, 
	  vmargin={25mm, 25mm},
	  headsep=10mm,
	  headheight=5mm,
	  footskip=10mm
	}
}

\begin{document}

\begin{center}

\LARGE Construction of 2-factors in the middle layer of the discrete cube
\vspace{2mm}

\Large Torsten Mütze \quad Franziska Weber
\vspace{2mm}

\large
  Institute of Theoretical Computer Science \\
  ETH Zürich, 8092 Zürich, Switzerland \\
  {\small\tt muetzet@inf.ethz.ch}, {\small\tt frweber@student.ethz.ch}
\vspace{5mm}

\small

\begin{minipage}{0.8\linewidth}
\textsc{Abstract.}
Define the middle layer graph as the graph whose vertex set consists of all bitstrings of length $2n+1$ that have exactly $n$ or $n+1$ entries equal to 1, with an edge between any two vertices for which the corresponding bitstrings differ in exactly one bit. In this work we present an inductive construction of a large family of 2-factors in the middle layer graph for all $n\geq 1$. We also investigate how the choice of certain parameters used in the construction affects the number and lengths of the cycles in the resulting 2-factor.
\end{minipage}

\end{center}

\vspace{5mm}

%\tableofcontents

\section{Introduction}

Consider the \emph{$n$-dimensional cube $Q_n$}, the graph with vertex set $\{0,1\}^n$ (the set of all bitstrings of length $n$) and an edge between any two vertices for which the corresponding bitstrings differ in exactly one bit. The cube has been studied extensively, and it is straightforward to exhibit e.g.\ perfect matchings or Hamiltonian cycles in this graph for all $n\geq 1$. The situation gets more involved if we consider subgraphs of $Q_n$, such as the graph induced by all vertices whose bitstrings contain exactly $k$ or $k+1$ entries equal to 1, where $0\leq k\leq n-1$. We denote this graph by $Q_n(k,k+1)$, and refer to it as a \emph{layer of $Q_n$}. The graph $Q_n(k,k+1)$ is clearly bipartite, and a straightforward application of Hall's theorem proves the existence of a matching that saturates all the vertices in the smaller of the two partition classes. However, it takes considerable effort to come up with explicit descriptions of such matchings \cite{MR0319772, MR0389608, MR962224, MR1268348}. The existence of a Hamiltonian path or cycle in the middle layer graph $Q_{2n+1}(n,n+1)$ for all $n\geq 1$ is asserted by the well-known (and still unproven) \emph{middle levels conjecture} (also known as revolving door conjecture).
An even more general conjecture due to Lov{\'a}sz~\cite{MR0263646} asserts that in fact every connected vertex-transitive graph contains a Hamiltonian path. The middle levels conjecture originated probably with Havel~\cite{MR737021} and Buck and Wiedemann~\cite{MR737262}, but has also been attributed to Dejter, Erd{\H{o}}s, Trotter~\cite{MR962224}, and various others.
With the help of a computer it has been verified that $Q_{2n+1}(n,n+1)$ indeed contains a Hamiltonian cycle for $n\leq 19$ \cite{MR2548541, shimada-amano}. It is also known that the middle layer graph contains a cycle that visits a $(1-o(1))$-fraction of all vertices~\cite{MR2046083}. Unfortunately, attempts to obtain a Hamiltonian cycle from the union of two perfect matchings have not been successful so far \cite{MR962223, MR962224}.

\subsection{Our results}

In this work we present an inductive construction of a large family of 2-factors in the middle layer graph $Q_{2n+1}(n,n+1)$ for all $n\geq 1$ (a 2-factor of a graph $G$ is a $2$-regular spanning subgraph, or equivalently, a family of disjoint cycles visiting all the vertices of $G$).
Our construction is parametrized by a sequence of parameters $(\alpha_{2i})_{1\leq i\leq n}$, where $\alpha_{2i}\in\{0,1\}^{i-1}$, and each of the $\prod_{i=1}^n 2^{i-1}=2^{\binom{n}{2}}=2^{\Theta(n^2)}$ different parameter sequences yields a different 2-factor in $Q_{2n+1}(n,n+1)$ (see Theorem~\ref{thm:different-2-factors} below).
\TM{Hier könnte man noch darauf eingehen, dass man durch Anwendung von Automorphismen potentiell noch mehr 2-factors bekommen könnte, aber an der Grössenordnung im Exponenten ändert dieser Faktor $n!$ nichts.}
For comparison, by combining two perfect matchings from the families of matchings described in~\cite{MR962224, MR1268348}, we only obtain at most $\Theta((n\cdot (2n)!)^2)=2^{\Theta(n\log n)}$ different 2-factors.
\TM{Man beachte, dass bei den Matchings die Vervielfachung durch Automorphismen schon gezählt wird.}

By changing the parameter sequence $(\alpha_{2i})_{1\leq i\leq n}$ we may control the number and lengths of the cycles in the resulting 2-factor. We prove that for \emph{any} choice of the parameter sequence the length of all cycles in the resulting 2-factor in $Q_{2n+1}(n,n+1)$ is a multiple of $4n+2$. In particular, the length of a shortest cycle is at least $4n+2$ (see Theorem~\ref{thm:all-alpha-cycles} below).
We also prove that for \emph{one particular} choice of the parameter sequence, the resulting 2-factor in $Q_{2n+1}(n,n+1)$ has $|\cT_{n+1}|$ many cycles, where $\cT_{n+1}$ denotes the set of all plane trees on $n+1$ vertices (we have $(|\cT_{n+1}|)_{n\geq 1}=(1,1,2,3,6,14,34,95,280,854,\ldots)$, see~\cite{plane-tree-seq}). For $n\geq 4$, the length of a shortest cycle in this 2-factor is $2(4n+2)$, the length of a longest cycle is $2n(4n+2)$, and a $(1-o(1))$-fraction of all cycles have length $2n(4n+2)$ (see Theorem~\ref{thm:alpha0-cycles} below).

When aiming for a 2-factor with few cycles (ideally only a single cycle, which would then be a Hamiltonian cycle), the advantage of our construction compared to simply combining two perfect matchings in the middle layer graph, is that the building blocks in our construction are paths, not just single edges. In fact, with the help of a computer we explored a small fraction of the parameter space for all $n\leq 14$ and thus found many Hamiltonian paths and cycles in $Q_{2n+1}(n,n+1)$ for those values of $n$ (see Section~\ref{sec:experiments} below). Those experiments suggest that the family of 2-factors arising from our construction is large enough to prove the middle levels conjecture for many more values of $n$, if not for infinitely many.

\subsection{Organization of this paper}

We begin by describing our construction in Section~\ref{sec:defs-construction}. The proof of a key lemma (Lemma~\ref{lemma:FL-invariant} below) which ensures that the construction works as claimed, is deferred to Section~\ref{sec:correctness}. In Section~\ref{sec:difference-2-factors} we analyze how different 2-factors arising from different parameter sequences are. In Section~\ref{sec:structure-2-factor} we investigate the number and lengths of the cycles in the 2-factors from our construction. In Section~\ref{sec:experiments} we briefly discuss the results of our computer experiments.

\section{Construction of a 2-factor in the middle layer of $Q_{2n+1}$}
\label{sec:defs-construction}

\subsection{Definitions and notation}
\label{sec:notation}

We start by collecting a few basic definitions that will be used throughout the paper.

\textit{Reversing, inverting and concatenating bitstrings.}
For any bitstring $x=(x_1,x_2,\ldots,x_n)$, $x_i\in\{0,1\}$, we define $\rev(x):=(x_n,x_{n-1},\ldots,x_1)$. Furthermore, setting $\ol{0}:=1$, $\ol{1}:=0$, we define $\ol{x}:=(\ol{x_1},\ol{x_2},\ldots,\ol{x_n})$.
For bitstrings $x$ and $y$ we denote by $x\circ y$ the concatenation of $x$ and $y$.
For any bitstring $x$ we define $x^0:=()$ and $x^k:=x\circ x^{k-1}$ for any integer $k\geq 1$. For a set of bitstrings $X$ and a bitstring $y$ we define $X\circ y:=\{x\circ y\mid x\in X\}$. We extend this notion to graphs $G$ whose vertex set is a set of bitstrings: For any bitstring $y$ we let $G\circ y$ denote the graph obtained from $G$ by attaching the bitstring $y$ to every vertex of $G$ (so we have $V(G\circ y)=V(G)\circ y$). Let $\cG$ be a family of graphs, all of whose vertex sets are sets of bitstrings. For any bitstring $y$ we define $\cG\circ y:=\{G\circ y\mid G\in\cG\}$, and for any set of bitstrings $Y$ we define $\cG\circ Y:=\{G\circ y\mid G\in\cG \wedge y\in Y\}$.

\textit{The discrete cube and its layers.}
For any integer $n\geq 1$ we let $B_n:=\{0,1\}^n$ denote the set of all bitstrings of length $n$. 
Recall that we defined the \emph{$n$-dimensional cube $Q_n$} as the graph with vertex set $B_n$ and an edge between any two vertices for which the corresponding bitstrings differ in exactly one bit.
For any integer $0\leq k\leq n$ we let $B_n(k)\seq B_n$ denote the set of all bitstrings of length $n$ with exactly $k$ entries equal to 1 (and the other $n-k$ entries equal to 0).
Recall that we defined the graph $Q_n(k,k+1)$, $0\leq k\leq n-1$, as the subgraph of $Q_n$ induced by the vertex sets $B_n(k)$ and $B_n(k+1)$, and that we refer to $Q_n(k,k+1)$ as a \emph{layer} of $Q_n$. In particular, we will refer to $Q_{2n}(k,k+1)$ for all $k=n,n+1,\ldots,2n-1$ as the \emph{upper layers} of $Q_{2n}$, and to $Q_{2n+1}(n,n+1)$ as the \emph{middle layer} of $Q_{2n+1}$.
% The number of vertices of $Q_{2n+1}(n,n+1)$ is $\binom{2n+1}{n}+\binom{2n+1}{n+1}=2\binom{2n+1}{n}=(1+o(1))\cdot 4^{n+1}/\sqrt{\pi n}=\Theta(4^n/\sqrt{n})$.

\textit{Inductive decomposition of the discrete cube.}
Beside the decomposition of $Q_n$ into layers, there is another important inductive decomposition of this graph. Note that $Q_n$ is obtained by taking two copies of $Q_{n-1}$, attaching a $0$ to all bitstrings in one copy (this yields a copy of the graph $Q_{n-1}\circ(0)$), attaching a $1$ to all bitstrings in the other copy (this yields a copy of the graph $Q_{n-1}\circ(1)$) and connecting corresponding vertices by a perfect matching $M_n$ (so the bitstrings corresponding to the end vertices of every edge of $M_n$ differ exactly in the last bit). Unrolling this inductive construction for another step, $Q_n$ is obtained from copies of $Q_{n-2}\circ(0,0)$, $Q_{n-2}\circ(1,0)$, $Q_{n-2}\circ(0,1)$ and $Q_{n-2}\circ(1,1)$ plus two perfect matchings $M_n$ and $M_n':=M_{n-1}\circ(0)\cup M_{n-1}\circ (1)$. As we shall see, our inductive construction of a 2-factor in the middle layer of $Q_{2n+1}$ is based on this inductive decomposition of $Q_{2n+2}$ into four copies of $Q_{2n}$ plus the two perfect matchings $M_{2n+2}$ and $M_{2n+2}'$ (see Figure~\ref{fig:cube-decomposition}).

\begin{figure}
\centering
\PSforPDF{
 \psfrag{b20}{$B_2(0)$}
 \psfrag{b21}{$B_2(1)$}
 \psfrag{b22}{$B_2(2)$}
 \psfrag{b40}{$B_4(0)$}
 \psfrag{b40}{$B_4(0)$}
 \psfrag{b41}{$B_4(1)$}
 \psfrag{b42}{$B_4(2)$}
 \psfrag{b43}{$B_4(3)$}
 \psfrag{b44}{$B_4(4)$}
 \psfrag{q5}{\Large $Q_2$}
 \psfrag{q6}{\Large $Q_4$}
 \psfrag{q1}{$Q_2\circ(0,0)$}
 \psfrag{q2}{$Q_2\circ(1,0)$}
 \psfrag{q3}{$Q_2\circ(0,1)$}
 \psfrag{q4}{$Q_2\circ(1,1)$}
 \psfrag{m30}{$M_3\circ(0)$}
 \psfrag{m31}{$M_3\circ(1)$}
 \psfrag{m4}{$M_4$}
 \psfrag{m4p}{$M_4'=M_3\circ(0)\cup M_3\circ(1)$}
 
 \psfrag{b2n2n}{\small $B_{2n}(2n)$}
 \psfrag{vdots}{$\vdots$}
 \psfrag{b2nnp1}{\small $B_{2n}(n+1)$}
 \psfrag{b2nn}{\small $B_{2n}(n)$}
 \psfrag{b2nnm1}{\small $B_{2n}(n-1)$}
 \psfrag{b2n0}{\small $B_{2n}(0)$}
 \psfrag{b2np22np2}{\small $B_{2n+2}(2n+2)$}
 \psfrag{b2np2np2}{\small $B_{2n+2}(n+2)$}
 \psfrag{b2np2np1}{\small $B_{2n+2}(n+1)$}
 \psfrag{b2np2n}{\small $B_{2n+2}(n)$}
 \psfrag{b2np20}{\small $B_{2n+2}(0)$}
 \psfrag{q2n}{\Large $Q_{2n}$}
 \psfrag{q2np2}{\Large $Q_{2n+2}$}
 \psfrag{q2n1}{$Q_{2n}\circ(0,0)$}
 \psfrag{q2n2}{$Q_{2n}\circ(1,0)$}
 \psfrag{q2n3}{$Q_{2n}\circ(0,1)$}
 \psfrag{q2n4}{$Q_{2n}\circ(1,1)$}
 \psfrag{m1}{$M_{2n+1}\circ(0)$}
 \psfrag{m2}{$M_{2n+1}\circ(1)$}
 \psfrag{m3}{$M_{2n+2}$}
 \psfrag{m2np2p}{$M_{2n+2}'=M_{2n+1}\circ(0)\cup M_{2n+1}\circ(1)$}
 \includegraphics{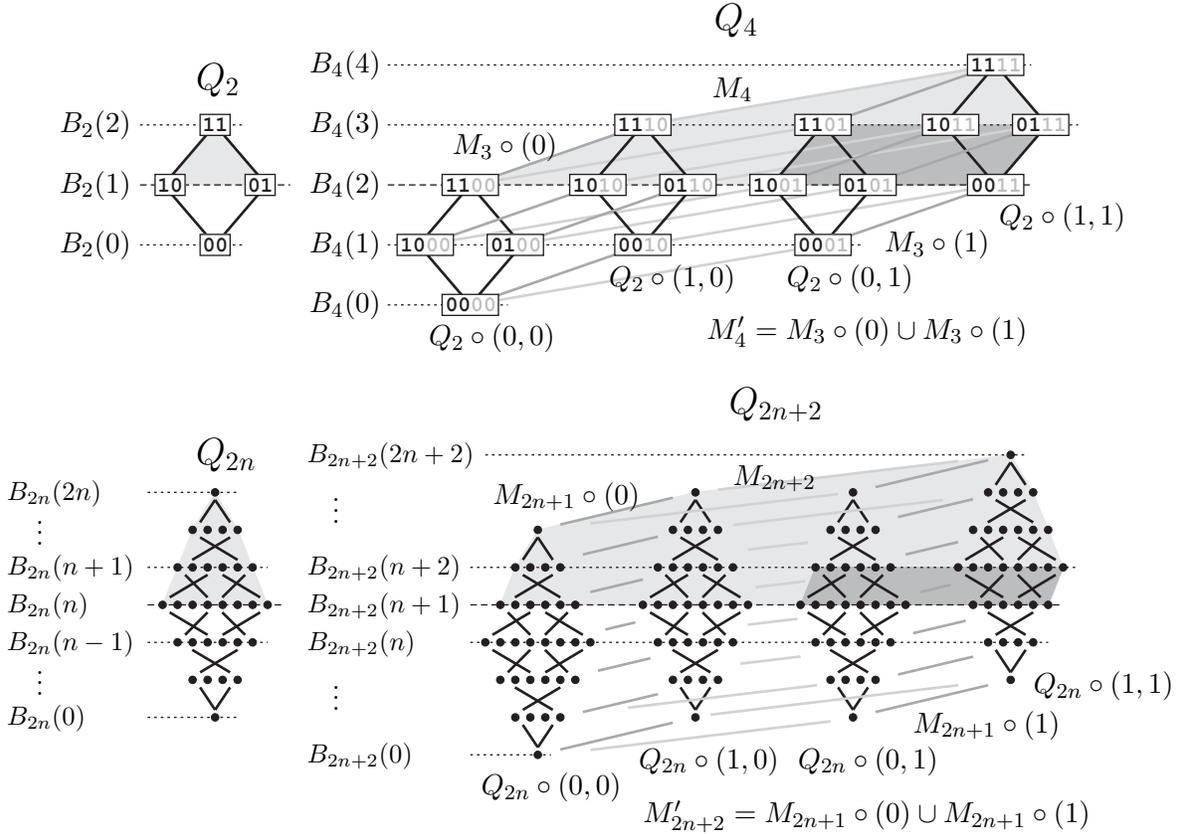}
}
\caption{Decomposition of $Q_{2n+2}$ into four copies of $Q_{2n}$ plus two perfect matchings (the top part shows a concrete example, the lower part a schematic representation of the general structure). The light grey regions show the upper layers of $Q_{2n}$ and $Q_{2n+2}$ and the dark grey region the middle layer of $Q_{2n+1}\circ(1)$.}
\label{fig:cube-decomposition}
\end{figure}

\textit{Oriented paths, dangling paths.}
In our approach we construct certain paths as subgraphs of layers of the cube. The order of vertices along those paths is important for us, i.e., $P=(v_1,v_2,\ldots,v_l)$ is a different \emph{oriented path} than $P'=(v_l,v_{l-1},\ldots,v_1)$.
For an oriented path $P=(v_1,v_2,\ldots,v_l)$ we define $F(P):=v_1$, $S(P):=v_2$ and $L(P):=v_l$, as the first, second and last vertex of $P$, respectively. We extend this notion to a family $\cP$ of oriented paths by setting $F(\cP):=\{F(P)\mid P\in\cP\}$, $S(\cP):=\{S(P)\mid P\in\cP\}$ and $L(\cP):=\{L(P)\mid P\in\cP\}$.

We refer to a path $P$ in $Q_n(k,k+1)$ that starts and ends at a vertex in the set $B_n(k)$ as a \emph{dangling path}. As $Q_n(k,k+1)$ is bipartite, every second vertex of such a path $P$ is contained in the set $B_n(k+1)$ (and $P$ has even length).

\subsection{Our construction}
\label{sec:construction}

The construction is parametrized by some sequence $(\alpha_{2i})_{i\geq 1}$, $\alpha_{2i}\in\{0,1\}^{i-1}$. Given this sequence, we inductively construct a family $\cP_{2n}(k,k+1)$ of disjoint dangling oriented paths in $Q_{2n}(k,k+1)$ for all $n\geq 1$ and all $k=n,n+1,\ldots,2n-1$ such that the following conditions hold:
\begin{enumerate}[(i)]
\item The paths in $\cP_{2n}(n,n+1)$ visit all vertices in the sets $B_{2n}(n+1)$ and $B_{2n}(n)$.
\item For $k=n+1,\ldots,2n-1$, the paths in $\cP_{2n}(k,k+1)$ visit all vertices in the set $B_{2n}(k+1)$, and the only vertices not visited in the set $B_{2n}(k)$ are exactly the elements in the set $S(\cP_{2n}(k-1,k))$.
\end{enumerate}
For simplicity we do not make the dependence of the families $\cP_{2n}(k,k+1)$ from the parameters $\alpha_2,\alpha_4,\ldots$ explicit, but we will discuss those dependencies in detail in Section~\ref{sec:dependence-alpha} below.

\textbf{Induction basis $n=1$ ($Q_2$):}
For the induction basis we define
\begin{equation} \label{eq:ind-base-P}
  \cP_{2}(1,2):=\{((1,0),(1,1),(0,1))\}
\end{equation}
(the family $\cP_{2}(1,2)$ consists only of a single oriented path on three vertices). It is easily checked that this family of paths in \emph{the} upper layer of $Q_2$ satisfies the conditions~(i) and (ii) (condition~(ii) is satisfied trivially).

\textbf{Induction step $n\rightarrow n+1$ ($Q_{2n}\rightarrow Q_{2n+2}$), $n\geq 1$:}
The inductive construction consists of two intermediate steps. For the reader's convenience those steps are illustrated in Figure~\ref{fig:construction}.

\textit{First intermediate step: Construction of a 2-factor in the middle layer of $Q_{2n+1}$.}
Using only the paths in the family $\cP_{2n}(n,n+1)$ and the parameter $\alpha_{2n}=(\alpha_{2n}(1),\ldots,\alpha_{2n}(n-1))\in\{0,1\}^{n-1}$ we first construct a 2-factor in the middle layer of $Q_{2n+1}$.

Note that the layer graphs $Q_{2n}(n,n+1)$ and $Q_{2n}(n-1,n)$ are isomorphic to each other. We define an isomorphism $f_{\alpha_{2n}}$ between these graphs as follows: Let $\pi_{\alpha_{2n}}$ denote the permutation on the set $B_{2n}=\{0,1\}^{2n}$ that swaps any two adjacent bits at positions $2i$ and $2i+1$ for all $i=1,\ldots,n-1$, if and only if $\alpha_{2n}(i)=1$, and that maps the bits at position $1$ and $2n$ to itself. If e.g.\ $\alpha_{2n}=(0,\ldots,0)$, then no bits are swapped and $\pi_{\alpha_{2n}}$ is simply the identity permutation. For any bitstring $x\in B_{2n}$ we then define
\begin{equation} \label{eq:f-alpha}
  f_{\alpha_{2n}}(x):=\ol{\rev(\pi_{\alpha_{2n}}(x))} \enspace.
\end{equation}
The fact that this mapping is indeed an isomorphism between the graphs $Q_{2n}(n,n+1)$ and $Q_{2n}(n-1,n)$ follows easily by observing that $\rev(\pi_{\alpha_{2n}}(\bullet))$ is an automorphism of the graph $Q_{2n}(n,n+1)$ (this mapping just permutes the bits).

\begin{figure}
\centering
\PSforPDF{
 \psfrag{b2n2n}{\scriptsize $B_{2n}(2n)$}
 \psfrag{p2n2nm12n}{\scriptsize $\cP_{2n}(2n-1,2n)$}
 \psfrag{vdots}{$\vdots$}
 \psfrag{b2nnp1}{\scriptsize $B_{2n}(n+1)$}
 \psfrag{p2nnnp1}{\scriptsize $\cP_{2n}(n,n+1)$}
 \psfrag{b2nn}{\scriptsize $B_{2n}(n)$}
 \psfrag{fp}{\scriptsize $f_{\alpha_{2n}}(\cP_{2n}(n,n+1))$}
 \psfrag{b2nnm1}{\scriptsize $B_{2n}(n-1)$}
 \psfrag{b2np22np2}{\scriptsize $B_{2n+2}(2n+2)$}
 \psfrag{p2np2np2np3}{\scriptsize $\cP_{2n+2}(n+2,n+3)$}
 \psfrag{b2np2np2}{\scriptsize $B_{2n+2}(n+2)$}
 \psfrag{p2np2np1np2}{\scriptsize $\cP_{2n+2}(n+1,n+2)$}
 \psfrag{b2np2np1}{\scriptsize $B_{2n+2}(n+1)$}
 \psfrag{q2n}{\Large $Q_{2n}$}
 \psfrag{q2np1}{\Large $Q_{2n+1}\circ(1)$}
 \psfrag{q2np2}{\Large $Q_{2n+2}$}
 \psfrag{q2n1}{$Q_{2n}\circ(0,0)$}
 \psfrag{q2n2}{$Q_{2n}\circ(1,0)$}
 \psfrag{q2n3}{$Q_{2n}\circ(0,1)$}
 \psfrag{q2n4}{$Q_{2n}\circ(1,1)$}
 \psfrag{m2}{$M_{2n+1}^{FL}\circ(1)$} 
 \psfrag{m1}{$M_{2n+2}^{S}$}
 \includegraphics{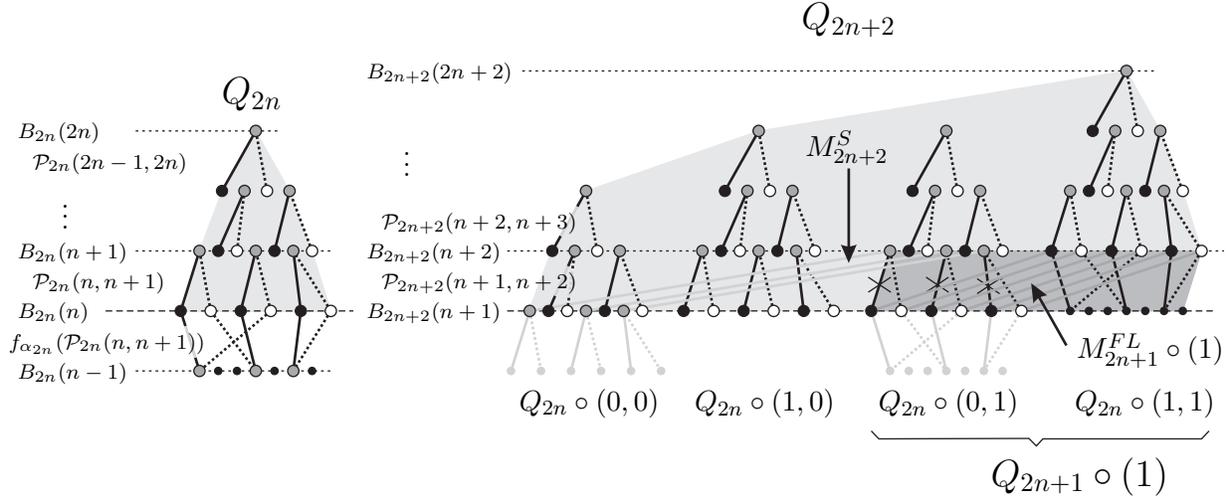}
}
\caption{Schematic illustration of the induction step. The light grey regions show the upper layers of $Q_{2n}$ and $Q_{2n+2}$ and the dark grey region the middle layer of $Q_{2n+1}\circ(1)$. For each dangling oriented path $P=(v_1,v_2,\ldots,v_l)$ contained in one of the families of paths in one of the layers, only the vertices $F(P)=v_1$ (black), $S(P)=v_2$ (grey) and $L(P)=v_l$ (white) are shown, and the path between the vertices $S(P)=v_2$ and $L(P)=v_l$ is represented by a dotted black line (even if this path contains more than one edge). The crossed-out edges are deleted from the 2-factor in the middle layer of $Q_{2n+1}\circ(1)$ to construct the paths in $\cP_{2n+2}(n+1,n+2)$.}
\label{fig:construction}
\end{figure}

We will later prove the following crucial lemma (see the left hand side of Figure~\ref{fig:construction}).

\begin{lemma} \label{lemma:FL-invariant}
For any $n\geq 1$ and any $\alpha_{2n}\in\{0,1\}^{n-1}$, we have
\begin{equation} \label{eq:FL-invariant}
  f_{\alpha_{2n}}(F(\cP_{2n}(n,n+1)))=F(\cP_{2n}(n,n+1)) \quad\text{and}\quad f_{\alpha_{2n}}(L(\cP_{2n}(n,n+1)))=L(\cP_{2n}(n,n+1)) \enspace,
\end{equation}
where $\cP_{2n}(n,n+1)$ is the family of paths in $Q_{2n}(n,n+1)$ constructed in previous steps for an arbitrary sequence of parameters $\alpha_2,\alpha_4,\ldots,\alpha_{2n-2}$, $\alpha_{2i}\in\{0,1\}^{i-1}$.
\end{lemma}

By the decomposition of $Q_{2n+1}$ into two copies of $Q_{2n}$ plus the perfect matching $M_{2n+1}$ described in Section~\ref{sec:notation}, the middle layer of $Q_{2n+1}$ can be decomposed into the graphs $Q_{2n}(n,n+1)\circ(0)$ and $Q_{2n}(n-1,n)\circ(1)$ plus the edges from $M_{2n+1}$ that connect the vertices in the set $B_{2n}(n)\circ(0)$ in the first graph to the vertices in the set $B_{2n}(n)\circ(1)$ in the second graph (see the right hand side of Figure~\ref{fig:construction}). Denoting by $M_{2n+1}^{FL}$ the edges from $M_{2n+1}$ that have one end vertex in the set $\big(F(\cP_{2n}(n,n+1))\cup L(\cP_{2n}(n,n+1))\big)\circ(0)\seq B_{2n}(n)\circ(0)$ (and the other in the set $\big(F(\cP_{2n}(n,n+1))\cup L(\cP_{2n}(n,n+1))\big)\circ(1)\seq B_{2n}(n)\circ(1)$), by Lemma~\ref{lemma:FL-invariant} the graph
\begin{equation} \label{eq:2-factor}
  \cC_{2n+1}:=\cP_{2n}(n,n+1)\circ(0)\cup f_{\alpha_{2n}}(\cP_{2n}(n,n+1))\circ(1)\cup M_{2n+1}^{FL}
\end{equation}
is a 2-factor in the middle layer of $Q_{2n+1}$, with the additional property that on every cycle of $\cC_{2n+1}$, every edge of the form $(F(P),S(P))\circ(0)$ for some $P\in\cP_{2n}(n,n+1)$ is oriented the same way. Even though we are eventually only interested in the 2-factor $\cC_{2n+1}$ defined in \eqref{eq:2-factor}, we need to specify how to proceed with the inductive construction of the families of paths $\cP_{2n+2}(k,k+1)$.

\textit{Second intermediate step: Splitting up the 2-factor into dangling paths.}
We now desribe how the families of paths $\cP_{2n+2}(k,k+1)$ for all $k=n+1,n+2,\ldots,2n+1$ satisfiying the conditions~(i) and (ii) are defined, using the previously constructed families $\cP_{2n}(k,k+1)$ and the 2-factor $\cC_{2n+1}$ defined in the first intermediate step.

Consider the decomposition of $Q_{2n+2}$ into copies of $Q_{2n}\circ(0,0)$, $Q_{2n}\circ(1,0)$, $Q_{2n}\circ(0,1)$ and $Q_{2n}\circ(1,1)$ plus the two perfect matchings $M_{2n+2}$ and $M_{2n+2}'$ as described in Section~\ref{sec:notation}.
For all $k=n+2,\ldots,2n+1$ we define
\begin{align}
  \cP_{2n+2}(k,k+1) &:= \cP_{2n}(k,k+1)\circ(0,0)\cup\cP_{2n}(k-1,k)\circ(1,0) \notag \\
                    &\qquad \cup \cP_{2n}(k-1,k)\circ(0,1)\cup\cP_{2n}(k-2,k-1)\circ(1,1) \label{eq:ind-step1-P} \enspace,
\end{align}
where we use the convention $\cP_{2n}(2n,2n+1):=\emptyset$ and $\cP_{2n}(2n+1,2n+2):=\emptyset$ to unify treatment of the two uppermost layers $\cP_{2n+2}(2n,2n+1)$ and $\cP_{2n+2}(2n+1,2n+2)$ (see Figure~\ref{fig:construction}). Note that so far none of the edges from the matchings $M_{2n+2}$ or $M_{2n+2}'$ is used.

The definition of the family $\cP_{2n+2}(n+1,n+2)$ is slightly more involved. Note that the layer graph $Q_{2n+2}(n+1,n+2)$ can be decomposed into $Q_{2n+1}(n+1,n+2)\circ(0)$ and $Q_{2n+1}(n,n+1)\circ(1)$ plus the edges from $M_{2n+2}$ that connect the vertices in the set $B_{2n+1}(n+1)\circ(0)$ in the first graph to the vertices in the set $B_{2n+1}(n+1)\circ(1)$ in the second graph. The first graph can be further decomposed into $Q_{2n}(n+1,n+2)\circ(0,0)$ and $Q_{2n}(n,n+1)\circ(1,0)$ plus some matching edges that are not relevant here. The second graph is the middle layer of $Q_{2n+1}\circ(1)$. Let $\cC_{2n+1}^-$ denote the graph obtained from the 2-factor $\cC_{2n+1}$  defined in \eqref{eq:2-factor} by removing every edge of the form $(F(P),S(P))\circ(0)$ for some $P\in\cP_{2n}(n,n+1)$ (those edges are crossed out in Figure~\ref{fig:construction}). As on every cycle of $\cC_{2n+1}$ every such edge is oriented the same way, $\cC_{2n+1}^-$ is a family of paths (visiting all vertices of the middle layer of $Q_{2n+1}$), with the property that each of those paths starts at a vertex of the form $S(P)\circ(0)$ and ends at a vertex of the form $F(P')\circ(0)$ for two (not necessarily distinct) paths $P,P'\in\cP_{2n}(n,n+1)$. Letting $M_{2n+2}^S$ denote the edges from $M_{2n+2}$ that have one end vertex in the set $S(\cP_{2n}(n,n+1))\circ(0,0)\seq B_{2n}(n+1)\circ(0,0)$ (and the other in the set $S(\cP_{2n}(n,n+1))\circ(0,1)\seq B_{2n}(n+1)\circ(0,1)$), it follows that
\begin{equation} \label{eq:new-paths}
  \cP_{2n+2}':=M_{2n+2}^S\cup \cC_{2n+1}^-\circ(1)
\end{equation}
is a family of dangling oriented paths, where we choose the orientation of each path such that the edge from the set $M_{2n+2}^S$ is the first edge (see Figure~\ref{fig:construction}).
Note that we have
\begin{subequations} \label{eq:new-paths-FSL}
\begin{align}
  F(\cP_{2n+2}') &= S(\cP_{2n}(n,n+1))\circ(0,0) \enspace, \label{eq:new-paths-F} \\
  S(\cP_{2n+2}') &= S(\cP_{2n}(n,n+1))\circ(0,1) \enspace, \label{eq:new-paths-S} \\
  L(\cP_{2n+2}') &= F(\cP_{2n}(n,n+1))\circ(0,1) \enspace. \label{eq:new-paths-L}
\end{align}
\end{subequations}
We then define
\begin{equation} \label{eq:ind-step2-P}
  \cP_{2n+2}(n+1,n+2):=\cP_{2n}(n+1,n+2)\circ(0,0)\cup\cP_{2n}(n,n+1)\circ(1,0)\cup \cP_{2n+2}' \enspace,
\end{equation}
where in the case $n=1$ we use the convention $\cP_2(2,3):=\emptyset$.

We now argue that the families of paths $\cP_{2n+2}(k,k+1)$, $k=n+1,n+2,\ldots,2n+1$, defined in \eqref{eq:ind-step1-P} and \eqref{eq:ind-step2-P} satisfy the conditions~(i) and (ii). For every $k=n+3,\ldots,2n+1$, by the definition in \eqref{eq:ind-step1-P} and by induction, the paths in $\cP_{2n+2}(k,k+1)$ visit all vertices in the set
\begin{equation*}
  B_{2n}(k+1)\circ(0,0)\cup B_{2n}(k)\circ(1,0)\cup B_{2n}(k)\circ(0,1)\cup B_{2n}(k-1)\circ(1,1)=B_{2n+2}(k+1) \enspace,
\end{equation*}
and the only vertices not visited in the set $B_{2n+2}(k)$ are exactly the elements in the set
\begin{multline*}
  S(\cP_{2n}(k-1,k))\circ(0,0)\cup S(\cP_{2n}(k-2,k-1))\circ(1,0) \\
  \cup S(\cP_{2n}(k-2,k-1))\circ(0,1)\cup S(\cP_{2n}(k-3,k-2))\circ(1,1) \enspace.
\end{multline*}
As for those $k$ the family of paths $\cP_{2n+2}(k-1,k)$ in the layer below is also defined via \eqref{eq:ind-step1-P}, this set is equal to $S(\cP_{2n+2}(k-1,k))$, proving that $\cP_{2n+2}(k,k+1)$ indeed satisfies condition~(ii).

By the definition in \eqref{eq:ind-step1-P} and by induction, the paths in the family $\cP_{2n+2}(n+2,n+3)$ visit all vertices in the set $B_{2n+2}(n+3)$, and the only vertices not visited in the set $B_{2n+2}(n+2)$ are exactly the elements in the set
\begin{equation*}
  S(\cP_{2n}(n+1,n+2))\circ(0,0)\cup S(\cP_{2n}(n,n+1))\circ(1,0)\cup S(\cP_{2n}(n,n+1))\circ(0,1) \enspace.
\end{equation*}
By the definition in \eqref{eq:ind-step2-P} and by \eqref{eq:new-paths-S} this set is equal to $S(\cP_{2n+2}(n+1,n+2))$, proving that $\cP_{2n+2}(n+2,n+3)$ indeed satisfies condition~(ii).

It remains to show that the family $\cP_{2n+2}(n+1,n+2)$ satisfies condition~(i). This follows directly from the definitions in \eqref{eq:new-paths} and \eqref{eq:ind-step2-P} and by induction, using that the paths in $\cC_{2n+1}^-\circ(1)$ visit \emph{all} vertices in the middle layer of $Q_{2n+1}\circ(1)$ (recall that those paths were obtained from a 2-factor in this graph), and that the only vertices in $Q_{2n}(n+1,n+2)\circ(0,0)$ not visited by the paths in $\cP_{2n}(n+1,n+2)\circ(0,0)$ are exactly the first vertices of the paths $\cP_{2n+2}'$ (cf.~\eqref{eq:new-paths-F}).

\subsection{Dependence on the parameter sequence}
\label{sec:dependence-alpha}

It follows inductively from our construction that for all $k=n,\ldots,2n-1$, the family of paths $\cP_{2n}(k,k+1)$ depends on all parameters $\alpha_2,\alpha_4,\ldots,\alpha_{2(2n-1-k)}$, and that the 2-factor $\cC_{2n+1}$ defined in \eqref{eq:2-factor} depends on all parameters $\alpha_2,\alpha_4,\ldots,\alpha_{2n}$.
As $\alpha_{2i}\in\{0,1\}^{i-1}$, our construction therefore yields at most $\prod_{i=1}^{n} 2^{i-1}=2^{\binom{n}{2}}$ different 2-factors in the middle layer of $Q_{2n+1}$. We will show later (see Theorem~\ref{thm:different-2-factors} below) that all of those 2-factors are indeed different.
The most interesting question is of course how the chosen parameter sequence affects the number and lengths of the cycles in the 2-factor $\cC_{2n+1}$ (see Section~\ref{sec:structure-2-factor} below). Of course, the number and/or the lengths of the cycles might be the same even for different parameter sequences.

Even though the paths in the families $\cP_{2n}(k,k+1)$ depend on the parameter sequence $(\alpha_{2i})_{i\geq 1}$, it follows from Lemma~\ref{lemma:FL-invariant} that the sets of first, second and last vertices of paths from those families do not depend on the sequence $(\alpha_{2i})_{i\geq 1}$. In particular, the number of paths in the families $\cP_{2n}(k,k+1)$ is independent of $(\alpha_{2i})_{i\geq 1}$ (those numbers are already fixed by the conditions~(i) and (ii) from Section~\ref{sec:construction} and the cardinalities of the sets $B_{2n}(k)$, $k=n,n+1,\ldots,2n$). Note moreover that the pairs $(F(P),S(P))$ for all paths $P\in\cP_{2n}(k,k+1)$ are the same regardless of the sequence $(\alpha_{2i})_{i\geq 1}$ (which last vertex $L(P)$ from the set of all last vertices belongs to this path does of course depend on the chosen parameter sequence). As we will see later, even the length of those paths is independent of the parameter sequence (see Lemma~\ref{lemma:all-alpha-paths} below).

\section{Correctness of the construction}
\label{sec:correctness}

In this section we prove Lemma~\ref{lemma:FL-invariant}, thus showing that our construction described in the last section indeed works as claimed. Our proof strategy is as follows: After setting up some machinery that relates bitstrings to another family of combinatorial objects, namely lattice paths, we consider an abstract recursion over sets of bitstrings and show that the solutions of this recursion correspond to certain families of lattice paths. It will then be easy to convince ourselves that the sets of first, second and last vertices of the oriented paths in the families $\cP_{2n}(k,k+1)$ arising in our construction satisfy exactly this abstract recursion, which allows us to apply our knowledge from the world of lattice paths and to derive Lemma~\ref{lemma:FL-invariant}.

\subsection{Bitstrings and lattice paths}
\label{sec:lattice-paths}

In this section we introduce some terminology related to lattice paths in $\mathbb{Z}^2$, explain the relation of those combinatorial objects to bitstrings (these are the vertex labels of $Q_n$ and thus the objects our construction works with), and establish an invariance property of certain families of lattice paths (Lemma~\ref{lemma:dyck-paths-invariant} below).

\textit{Various families of lattice paths.}
For any integer $n\geq 0$ we denote by $P_n$ the set of lattice paths in $\mathbb{Z}^2$ that start at $(0,0)$ and move $n$ steps, each of which changes the current coordinate by either $(+1,+1)$ or $(+1,-1)$. We refer to such a step as an upstep or downstep, respectively.
For any integers $n,k\geq 0$ we let $D_n(k)$ denote the set of lattice paths from $P_n$ that never move below the line $y=0$ and that have exaktly $k$ upsteps\footnote{Our notation is motivated by the fact that the lattice paths in the set $D_{2n}(n)$, which start at $(0,0)$, end at $(2n,0)$ and never move below the line $y=0$, are commonly known as \emph{Dyck paths} in the literature.}.
% named after the German mathematician Walther von Dyck (1856--1934)
Note that such a path has $n-k$ downsteps and therefore ends at $(n,2k-n)$.
We define $D_n^{>0}(k)\seq D_n(k)$ as the set of lattice paths that have no point of the form $(x,0)$, $1\leq x\leq n$, and $D_n^{=0}(k)\seq D_n(k)$ as the set of lattice paths that have at least one point of the form $(x,0)$, $1\leq x\leq n$. We clearly have $D_n(k)=D_n^{=0}(k)\cup D_n^{>0}(k)$.
Furthermore, we let $D_n^-(k)$ denote the set of lattice paths from $P_n$ that move below $y=0$ exactly once and that have exactly $k$ upsteps. Note that such a path has exactly one point of the form $(x,-1)$ and ends at $(n,2k-n)$.
Depending on the values of $n$ and $k$ the sets of lattice paths we just defined might be empty. E.g., we have $D_{2n}^{>0}(n)=\emptyset$ and therefore $D_{2n}(n)=D_{2n}^{=0}(n)$.

Given two lattice paths $p$ and $q$, we denote by $p\circ q$ the lattice path obtained by gluing the first point of $q$ onto the last point of $p$ (the first point of $p\circ q$ is the same as the first point of $p$). For a set of lattice paths $P$ and a lattice path $q$ we define $P\circ q:=\{p\circ q\mid p\in P\}$.
We sometimes identify a lattice path $p\in P_n$ with its step sequence $p=(p_1,\ldots,p_n)$, $p_i\in\{\nearrow,\searrow\}$, where $p_i=\nearrow$ if the $i$-th step of $p$ is an upstep and $p_i=\searrow$ if the $i$-th step of $p$ is a downstep.
Using these notations we clearly have for $*\in\{=0,-\}$, all $n\geq 1$ and all $k=n+2,\ldots,2n+1$ that
\begin{subequations} \label{eq:D-partitions}
\begin{align}
  D_{2n+2}^*(k) &= D_{2n}^*(k)\circ (\searrow,\searrow)
                 \cup D_{2n}^*(k-1)\circ (\nearrow,\searrow) \notag \\
                 &\qquad \cup D_{2n}^*(k-1)\circ (\searrow,\nearrow)
                               \cup D_{2n}^*(k-2)\circ (\nearrow,\nearrow) \enspace, \label{eq:D2np2-*-u-partition} \\
  D_{2n+2}^{>0}(k+1) &= D_{2n}^{>0}(k+1)\circ (\searrow,\searrow)
                 \cup D_{2n}^{>0}(k)\circ (\nearrow,\searrow) \notag \\
                 &\qquad \cup D_{2n}^{>0}(k)\circ (\searrow,\nearrow)
                               \cup D_{2n}^{>0}(k-1)\circ (\nearrow,\nearrow) \enspace. \label{eq:D2np2-g0-u-partition}
\end{align}
Similarly, for all $n\geq 0$ we have
\begin{align}
  D_{2n+2}^{=0}(n+1) &= \big(D_{2n}^{=0}(n+1)\cup D_{2n}^{>0}(n+1)\big) \circ (\searrow,\searrow) \cup
                             D_{2n}^{=0}(n)\circ (\nearrow,\searrow) \enspace, \label{eq:D2np2-eq0-m-partition} \\
  D_{2n+2}^{>0}(n+2) &= D_{2n}^{>0}(n+2)\circ (\searrow,\searrow) \cup
                        D_{2n}^{>0}(n+1)\circ (\nearrow,\searrow) \cup
                        D_{2n}^{>0}(n+1)\circ (\searrow,\nearrow) \enspace, \label{eq:D2np2-g0-m-partition} \\
  D_{2n+2}^-(n+1) &= D_{2n}^-(n+1)\circ (\searrow,\searrow) \cup
                     D_{2n}^-(n)\circ (\nearrow,\searrow) \cup
                     D_{2n}^{=0}(n)\circ (\searrow,\nearrow) \enspace. \label{eq:D2np2-m-m-partition}
\end{align}
\end{subequations}
Note that all the unions in \eqref{eq:D-partitions} are disjoint and that some of the sets participating in the unions might be empty.

\textit{Bijection $\varphi$ between bitstrings and lattice paths.}
For any $x\in B_n=\{0,1\}^n$, $x=(x_1,\ldots,x_n)$, we define $\varphi(x)$ as the lattice path from $P_n$ whose $i$-th step is an upstep if $x_i=1$ and a downstep if $x_i=0$. Note that the step sequence of $\varphi(x)$ is obtained from $(x_1,\ldots,x_n)$ by replacing every entry $1$ by $\nearrow$ and every entry $0$ by $\searrow$. This mapping is clearly a bijection between $B_n$ and $P_n$.

We extend the operation of reversing and inverting a bitstring to lattice paths by defining the mapping $\ol{\rev}:P_n\rightarrow P_n$ as
\begin{equation} \label{eq:rev-paths}
  \ol{\rev}:=\varphi\circ\ol{\rev}\circ\varphi^{-1}
\end{equation}
(we write the composition of mappings $g$ and $h$ as $g\circ h$, where $(g\circ h)(x):=g(h(x))$).
Note that $\ol{\rev}$ as defined in \eqref{eq:rev-paths} simply mirrors every lattice path from the set $P_{2n}$ with endpoint $(2n,0)$ along the axis $x=n$.

In a similar fashion we also extend the mappings $\pi_{\alpha_{2n}}$ and $f_{\alpha_{2n}}$, defined around \eqref{eq:f-alpha} as mappings on the set $B_{2n}$, to mappings on the set $P_{2n}$, by defining for any $\alpha_{2n}\in\{0,1\}^{n-1}$
\begin{equation} \label{eq:pi-paths}
  \pi_{\alpha_{2n}} := \varphi\circ \pi_{\alpha_{2n}}\circ\varphi^{-1}
\end{equation}
and
\begin{equation} \label{eq:f-alpha-paths}
  f_{\alpha_{2n}} := \varphi\circ f_{\alpha_{2n}} \circ\varphi^{-1}
    \eqBy{eq:f-alpha} \varphi\circ \ol{\rev}\circ\pi_{\alpha_{2n}} \circ\varphi^{-1}
    \eqByM{\eqref{eq:rev-paths},\eqref{eq:pi-paths}} \ol{\rev}\circ\pi_{\alpha_{2n}} \enspace.
\end{equation}
Note that $\pi_{\alpha_{2n}}$ as defined in \eqref{eq:pi-paths} swaps the order of any two adjacent steps $2i$ and $2i+1$, $i=1,\ldots,n-1$, of a given lattice path from the set $P_{2n}$, if and only if $\alpha_{2n}(i)=1$.

\begin{lemma} \label{lemma:dyck-paths-invariant}
For any $n\geq 1$ and any $\alpha_{2n}\in\{0,1\}^{n-1}$ the mapping $f_{\alpha_{2n}}:P_{2n}\rightarrow P_{2n}$ defined in \eqref{eq:f-alpha-paths} maps each of the sets $D_{2n}^{=0}(n)$ and $D_{2n}^-(n)$ onto itself. Furthermore, for any lattice path $p\in D_{2n}^-(n)$, denoting by $x$ and $x'$ the abscissas where $p$ and $f_{\alpha_{2n}}(p)$ touch the line $y=-1$, respectively, we have $x'=2n-x$.
\end{lemma}

Note that even though the sets $D_{2n}^{=0}(n)$ and $D_{2n}^-(n)$ are invariant under the mapping $f_{\alpha_{2n}}$, changing the parameter $\alpha_{2n}$ will of course change the images of certain lattice paths from those sets.

\begin{proof}
The first part of the lemma follows from \eqref{eq:f-alpha-paths} if we can show that each of the mappings $\overline{\rev}:P_{2n}\rightarrow P_{2n}$ and $\pi_{\alpha_{2n}}:P_{2n}\rightarrow P_{2n}$ maps each of the sets $D_{2n}^{=0}(n)$ and $D_{2n}^-(n)$ onto itself.

For the mapping $\overline{\rev}$ this is trivial, as $\overline{\rev}$ simply mirrors every lattice path from the set $P_{2n}$ with endpoint $(2n,0)$ along the axis $x=n$.

Note that the mapping $\pi_{\alpha_{2n}}$ leaves the $y$-coordinates of a given lattice path at all \emph{odd} abscissas $x=1,3,\ldots,2n-1$ invariant, and decreases the $y$-coordinates at all \emph{even} abscissas $x=2i$, $i=1,\ldots,n-1$, by $-2$ if and only if $\alpha_{2n}(i)=1$ and the steps $2i$ and $2i+1$ of the path are an upstep and a downstep, respectively. This mapping clearly leaves the $y$-coordinates at the abscissas $x=0$ and $x=2n$ invariant as well.

Observe that for every lattice path from the set $D_{2n}^{=0}(n)$ or from the set $D_{2n}^-(n)$, the $y$-coordinates at all odd abscissas $x=1,3,\ldots,2n-1$ are odd, and the $y$-coordinates at all even abscissas $x=0,2,\ldots,2n$ are even (in particular, the abscissa where a lattice path from the set $D_{2n}^-(n)$ touches the line $y=-1$ is odd). This property implies that for any pair $2i$ and $2i+1$, $1\leq i\leq n-1$, of an upstep and a downstep on such a path, the point $(2i,y')$ on the path satisfies $y'\geq 2$ ($y'$ must be even, and if it were 0 or less, then this path would have at least two points with a negative $y$-coordinate). Using these observations and the above-mentioned properties how the mapping $\pi_{\alpha_{2n}}$ affects the $y$-coordinates at the odd and even abscissas, it follows that this mapping indeed maps each of the sets $D_{2n}^{=0}(n)$ and $D_{2n}^-(n)$ onto itself. This proves the first part of the lemma.

The second part of the lemma follows immediately from the observation that the point $(x,-1)$, $0\leq x\leq 2n$, on a lattice path $p\in D_{2n}^-(n)$ must have an odd abscissa and is therefore invariant under the mapping $\pi_{\alpha_{2n}}$.
\end{proof}

\subsection{An abstract recursion}

In this section we define an abstract recursion over sets of bitstrings and show that the solutions of this recursion correspond to certain families of lattice paths (Lemma~\ref{lemma:FSL-D-isomorphic} below).

For all $n\geq 1$ and all $k=n,n+1,\ldots,2n-1$ we define sets of bitstrings $F_{2n}(k,k+1)\seq B_{2n}(k)$, $S_{2n}(k,k+1)\seq B_{2n}(k+1)$ and $L_{2n}(k,k+1)\seq B_{2n}(k)$ recursively as follows:

For $n=1$ we define
\begin{equation} \label{eq:ind-base-FSL}
  F_2(1,2) := \{(1,0)\} \enspace, \quad
  S_2(1,2) := \{(1,1)\} \enspace, \quad
  L_2(1,2) := \{(0,1)\} \enspace.
\end{equation}

For any $n\geq 1$ and all $k=n+2,\ldots,2n+1$ we define
\begin{subequations} \label{eq:ind-step1-FSL}
\begin{align}
  F_{2n+2}(k,k+1) &:= F_{2n}(k,k+1)\circ(0,0)\cup F_{2n}(k-1,k)\circ(1,0) \notag \\
                  &\qquad \cup F_{2n}(k-1,k)\circ(0,1)\cup F_{2n}(k-2,k-1)\circ(1,1) \enspace, \label{eq:ind-step1-F} \\
  S_{2n+2}(k,k+1) &:= S_{2n}(k,k+1)\circ(0,0)\cup S_{2n}(k-1,k)\circ(1,0) \notag \\
                  &\qquad \cup S_{2n}(k-1,k)\circ(0,1)\cup S_{2n}(k-2,k-1)\circ(1,1) \enspace, \label{eq:ind-step1-S} \\
  L_{2n+2}(k,k+1) &:= L_{2n}(k,k+1)\circ(0,0)\cup L_{2n}(k-1,k)\circ(1,0) \notag \\
                  &\qquad \cup L_{2n}(k-1,k)\circ(0,1)\cup L_{2n}(k-2,k-1)\circ(1,1) \enspace, \label{eq:ind-step1-L}
\end{align}
\end{subequations}
where we use the convention $F_{2n}(2n,2n+1):=S_{2n}(2n,2n+1):=L_{2n}(2n,2n+1):=\emptyset$ and $F_{2n}(2n+1,2n+2):=S_{2n}(2n+1,2n+2):=L_{2n}(2n+1,2n+2):=\emptyset$.

Furthermore, for any $n\geq 1$ we define
\begin{subequations} \label{eq:ind-step2-FSL}
\begin{align}
  F_{2n+2}(n+1,n+2) &:= F_{2n}(n+1,n+2)\circ(0,0)\cup F_{2n}(n,n+1)\circ(1,0)\cup S_{2n}(n,n+1)\circ(0,0) \enspace, \label{eq:ind-step2-F} \\
  S_{2n+2}(n+1,n+2) &:= S_{2n}(n+1,n+2)\circ(0,0)\cup S_{2n}(n,n+1)\circ(1,0)\cup S_{2n}(n,n+1)\circ(0,1) \enspace, \label{eq:ind-step2-S} \\
  L_{2n+2}(n+1,n+2) &:= L_{2n}(n+1,n+2)\circ(0,0)\cup L_{2n}(n,n+1)\circ(1,0)\cup F_{2n}(n,n+1)\circ(0,1) \enspace, \label{eq:ind-step2-L}
\end{align}
\end{subequations}
where in the case $n=1$ we use the convention $F_2(2,3):=S_2(2,3):=L_2(2,3):=\emptyset$.

\begin{lemma} \label{lemma:FSL-D-isomorphic}
For any $n\geq 1$ and all $k=n,n+1,\ldots,2n-1$ we have
\begin{subequations}
\begin{align}
  \varphi(F_{2n}(k,k+1)) &= D_{2n}^{=0}(k) \enspace, \label{eq:F-D-isomorphic} \\
  \varphi(S_{2n}(k,k+1)) &= D_{2n}^{>0}(k+1) \enspace, \label{eq:S-D-isomorphic} \\
  \varphi(L_{2n}(k,k+1)) &= D_{2n}^-(k) \enspace, \label{eq:L-D-isomorphic}
\end{align}
\end{subequations}
where the sets $F_{2n}(k,k+1)$, $S_{2n}(k,k+1)$ and $L_{2n}(k,k+1)$ are defined in \eqref{eq:ind-base-FSL}, \eqref{eq:ind-step1-FSL} and \eqref{eq:ind-step2-FSL}.
\end{lemma}

Note that all unions in \eqref{eq:ind-step1-FSL} and \eqref{eq:ind-step2-FSL} are disjoint: This is obvious for the definitions in \eqref{eq:ind-step1-FSL}, \eqref{eq:ind-step2-S} and \eqref{eq:ind-step2-L}, as the two-bit strings appended to the bitstrings in the sets participating in each of the unions are distinct. For the definition in \eqref{eq:ind-step2-F} this follows from Lemma~\ref{lemma:FSL-D-isomorphic}, noting that by \eqref{eq:F-D-isomorphic} and \eqref{eq:S-D-isomorphic} the sets $F_{2n}(n+1,n+2)$ and $S_{2n}(n,n+1)$ participating in the union correspond to the sets $D_{2n}^{=0}(n+1)$ and $D_{2n}^{>0}(n+1)$ and are therefore disjoint.

\begin{proof}
We argue by induction over $n$. The fact that all three claimed relations hold for $n=1$ follows immediately from \eqref{eq:ind-base-FSL}.
For the induction step let $n\geq 1$ be fixed. We prove that the statement of the lemma holds for $n+1$ assuming that it holds for $n$. We distinguish the cases $n+2\leq k\leq 2n+1$ and $k=n+1$.

For $k=n+2,\ldots,2n+1$ we have
\begin{equation*}
\begin{split}
  \varphi(F_{2n+2}(k,k+1)) &\eqBy{eq:ind-step1-F}
    \varphi(F_{2n}(k,k+1)\circ(0,0)) \cup
    \varphi(F_{2n}(k-1,k)\circ(1,0)) \\
    &\qquad \cup
    \varphi(F_{2n}(k-1,k)\circ(0,1)) \cup
    \varphi(F_{2n}(k-2,k-1))\circ(1,1)) \\
  &\eqBy{eq:F-D-isomorphic}
    D_{2n}^{=0}(k)\circ(\searrow,\searrow) \cup
    D_{2n}^{=0}(k-1)\circ(\nearrow,\searrow) \\
    &\qquad \cup
    D_{2n}^{=0}(k-1)\circ(\searrow,\nearrow) \cup
    D_{2n}^{=0}(k-2)\circ(\nearrow,\nearrow)
  \eqBy{eq:D2np2-*-u-partition} D_{2n+2}^{=0}(k) \enspace,
\end{split}
\end{equation*}
where we used the induction hypothesis in the second step.
The proof that also the last two relations stated in the lemma hold in this case goes along very similar lines, using \eqref{eq:ind-step1-S}, \eqref{eq:S-D-isomorphic} and \eqref{eq:D2np2-g0-u-partition} in the first, second and third step, or \eqref{eq:ind-step1-L}, \eqref{eq:L-D-isomorphic} and \eqref{eq:D2np2-*-u-partition}, respectively. We omit the details here.

For the case $k=n+1$ we obtain
\begin{equation*}
\begin{split}
  \varphi(& F_{2n+2}(n+1,n+2)) \\
  &\eqBy{eq:ind-step2-F} 
    \varphi(F_{2n}(n+1,n+2)\circ(0,0)) \cup
    \varphi(F_{2n}(n,n+1)\circ(1,0)) \cup
    \varphi(S_{2n}(n,n+1))\circ(0,0)) \\
  &\eqByM{\eqref{eq:F-D-isomorphic},\eqref{eq:S-D-isomorphic}}
    \big(D_{2n}^{=0}(n+1) \cup D_{2n}^{>0}(n+1)\big)\circ(\searrow,\searrow) \cup
    D_{2n}^{=0}(n)\circ(\nearrow,\searrow)
  \eqBy{eq:D2np2-eq0-m-partition} D_{2n+2}^{=0}(n+1) \enspace,
\end{split}
\end{equation*}
where we used the induction hypothesis in the second step.
In a similar fashion we obtain
\begin{equation*}
\begin{split}
  \varphi(& S_{2n+2}(n+1,n+2)) \\
  &\eqBy{eq:ind-step2-S}
    \varphi(S_{2n}(n+1,n+2)\circ(0,0)) \cup
    \varphi(S_{2n}(n,n+1)\circ(1,0)) \cup
    \varphi(S_{2n}(n,n+1)\circ(0,1)) \\
  &\eqBy{eq:S-D-isomorphic}
    D_{2n}^{>0}(n+2)\circ(\searrow,\searrow) \cup
    D_{2n}^{>0}(n+1) \circ(\nearrow,\searrow) \cup
    D_{2n}^{>0}(n+1)\circ(\searrow,\nearrow)
  \eqBy{eq:D2np2-g0-m-partition} D_{2n+2}^{>0}(n+2)
\end{split}
\end{equation*}
and
\begin{equation*}
\begin{split}
  \varphi(& L_{2n+2}(n+1,n+2)) \\
  &\eqBy{eq:ind-step2-L}
    \varphi(L_{2n}(n+1,n+2)\circ(0,0)) \cup
    \varphi(L_{2n}(n,n+1)\circ(1,0)) \cup
    \varphi(F_{2n}(n,n+1)\circ(0,1)) \\
  &\eqByM{\eqref{eq:F-D-isomorphic},\eqref{eq:L-D-isomorphic}}
    D_{2n}^-(n+1)\circ(\searrow,\searrow) \cup
    D_{2n}^-(n)\circ(\nearrow,\searrow) \cup
    D_{2n}^{=0}(n)\circ(\searrow,\nearrow)
  \eqBy{eq:D2np2-m-m-partition} D_{2n+2}^-(n+1) \enspace.
\end{split}
\end{equation*}
This completes the proof.
\end{proof}

\subsection{Proof of Lemma~\texorpdfstring{\ref{lemma:FL-invariant}}{1}}

We introduce the abbreviations
\begin{subequations} \label{eq:FSL-abbreviation}
\begin{align}
  F_{2n}(k,k+1) &:= F(\cP_{2n}(k,k+1)) \enspace, \\
  S_{2n}(k,k+1) &:= S(\cP_{2n}(k,k+1)) \enspace, \\
  L_{2n}(k,k+1) &:= L(\cP_{2n}(k,k+1))
\end{align}
\end{subequations}
for the sets of first, second and last vertices of the oriented paths in the families $\cP_{2n}(k,k+1)$ arising in our construction.

\begin{proof}
Observe that the sets $F_{2n}(k,k+1)$, $S_{2n}(k,k+1)$ and $L_{2n}(k,k+1)$ defined in \eqref{eq:FSL-abbreviation} satisfy exactly the recursive relations in \eqref{eq:ind-base-FSL}, \eqref{eq:ind-step1-FSL} and \eqref{eq:ind-step2-FSL} (recall that those sets are independent of the parameter sequence $(\alpha_{2i})_{i\geq 1}$ used in our construction): This can be seen by comparing \eqref{eq:ind-base-P} with \eqref{eq:ind-base-FSL}, \eqref{eq:ind-step1-P} with \eqref{eq:ind-step1-FSL} and finally \eqref{eq:ind-step2-P} with \eqref{eq:ind-step2-FSL}, in the last step also using \eqref{eq:new-paths-FSL}.

We may thus apply Lemma~\ref{lemma:FSL-D-isomorphic}, and using the relations \eqref{eq:F-D-isomorphic} and \eqref{eq:L-D-isomorphic} for $k=n$, we obtain that proving \eqref{eq:FL-invariant} is equivalent to proving that the mapping $f_{\alpha_{2n}}$ defined in \eqref{eq:f-alpha-paths} satisfies
\begin{equation*}
  f_{\alpha_{2n}}(D_{2n}^{=0}(n))=D_{2n}^{=0}(n) \quad\text{and}\quad f_{\alpha_{2n}}(D_{2n}^-(n))=D_{2n}^-(n) \enspace,
\end{equation*}
which is exactly the assertion of Lemma~\ref{lemma:dyck-paths-invariant}.
\end{proof}

\begin{remark} \label{remark:use-lemma-FSL-D}
Using the abbreviations defined in \eqref{eq:FSL-abbreviation} we may and will from now on use Lemma~\ref{lemma:FSL-D-isomorphic} as a statement about the sets of first, second and last vertices of the oriented paths in the families $\cP_{2n}(k,k+1)$ arising in our construction (rather than as a statement about abstractly defined sets of bitstrings).
\end{remark}

\begin{remark}
It is not hard to deduce from the proof of Lemma~\ref{lemma:FL-invariant} and Lemma~\ref{lemma:dyck-paths-invariant} that the mappings $f_{\alpha_{2n}}$ defined in \eqref{eq:f-alpha} and parametrized by $\alpha_{2n}\in\{0,1\}^{n-1}$ are in fact the only isomorphisms between the graphs $Q_{2n}(n,n+1)$ and $Q_{2n}(n-1,n)$ that satisfy the invariance condition in \eqref{eq:FL-invariant} which is crucial for our construction. In fact, these mappings are even the only isomorphisms satisfying the slightly weaker invariance condition
\begin{equation*}
  f_{\alpha_{2n}}\big(F(\cP_{2n}(n,n+1))\cup L(\cP_{2n}(n,n+1))\big)=F(\cP_{2n}(n,n+1))\cup L(\cP_{2n}(n,n+1))
\end{equation*}
which could potentially also be exploited for the construction. In this sense our parametrization already captures the maximum possible freedom inherent in the construction.
\end{remark}

\section{How different are the 2-factors from different parameter sequences?}
\label{sec:difference-2-factors}

In this section we first show that different parameter sequences used in our construction indeed yield different 2-factors in the middle layer of $Q_{2n+1}$ (Theorem~\ref{thm:different-2-factors} below). We then consider the question which 2-factors obtained from our construction are mapped onto each other under automorphisms of $Q_{2n+1}(n,n+1)$ (Proposition~\ref{prop:automorphisms} below).

\begin{theorem} \label{thm:different-2-factors}
For any $n\geq 1$ and any two different parameter sequences $(\alpha_{2i})_{1\leq i\leq n}$, $(\alpha_{2i}')_{1\leq i\leq n}$, $\alpha_{2i},\alpha_{2i}'\in\{0,1\}^{i-1}$, the 2-factors $\cC_{2n+1}$ and $\cC_{2n+1}'$ defined in Section~\eqref{sec:construction} for these parameter sequences, respectively, are different subgraphs of $Q_{2n+1}(n,n+1)$.
\end{theorem}

\begin{proof}
For the reader's convenience, Figure~\ref{fig:different-2-factors} illustrates the notations used in the proof.

\begin{figure}
\centering
\PSforPDF{
 \psfrag{q2n1}{\Large $Q_{2n}\circ(0)$}
 \psfrag{q2n2}{\Large $Q_{2n}\circ(1)$}
 \psfrag{q2n3}{\Large $Q_{2n+1}$}
 \psfrag{m1}{$M_{2n+2}^S$}
 \psfrag{m2}{$M_{2n+1}^{FL}$}
 \psfrag{p2n2}{$P_{2n}(n,n+1)\circ(0)$}
 \psfrag{bnp1}{$B_{2n}(n)\circ(0)$}
 \psfrag{bnp2}{$B_{2n}(n)\circ(1)$}
 \psfrag{e}{$e$}
 \psfrag{c}{$C$}
 \psfrag{cp}{$C'$}
 \psfrag{v}{$v=f_{\alpha_{2n}}(F(P))\circ(1)$}
 \psfrag{fp1}{$f_{\alpha_{2n}}(P)\circ(0)$}
 \psfrag{fp2}{$f_{\alpha_{2n}}(P)\circ(1)$}
 \psfrag{fpp1}{$f_{\alpha_{2n}'}(P')\circ(0)$}
 \psfrag{fpp2}{$f_{\alpha_{2n}'}(P')\circ(1)$}
 \psfrag{fpp2n21}{$f_{\alpha_{2n}'}(\cP_{2n}(n,n+1))\circ(0)$}
 \psfrag{fpp2n22}{$f_{\alpha_{2n}'}(\cP_{2n}(n,n+1))\circ(1)$}
 \includegraphics{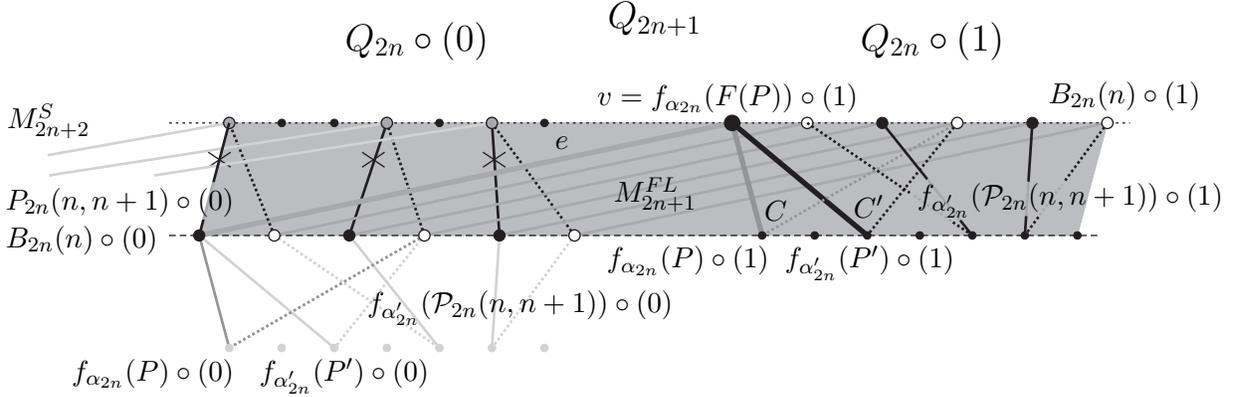}
}
\caption{Notations used in the proof of Theorem~\ref{thm:different-2-factors}.}
\label{fig:different-2-factors}
\end{figure}

We first assume that the sequences $(\alpha_{2i})_{1\leq i\leq n}$ and $(\alpha_{2i}')_{1\leq i\leq n}$ differ only in their respective last entry $\alpha_{2n}$ and $\alpha_{2n}'$, and show that the resulting 2-factors $\cC_{2n+1}$ and $\cC_{2n+1}'$ are different. We then show that this difference propagates through all further construction steps (when both parameter sequences are extended arbitrarily), which is enough to prove the statement of the lemma in full generality.

So let $\cP_{2n}(k,k+1)$, $k=n,n+1,\ldots,2n-1$, denote the families of oriented paths as defined by our construction from Section~\ref{sec:construction} for the parameter sequence $(\alpha_{2i})_{1\leq i\leq n-1}=(\alpha_{2i}')_{1\leq i\leq n-1}$. By Lemma~\ref{lemma:FSL-D-isomorphic} we have $\varphi(S(\cP_{2n}(n,n+1)))=D_{2n}^{>0}(n+1)$ (recall Remark~\ref{remark:use-lemma-FSL-D}), implying that there is a path $P\in\cP_{2n}(n,n+1)$ with $S(P)=(1)\circ(1,0)^{n-1}\circ(1)$. Note that $P$ satisfies the condition
\begin{equation} \label{eq:images-S-P}
  f_{\alpha_{2n}}(S(P))\notin f_{\alpha_{2n}'}(S(\cP_{2n}(n,n+1)))
\end{equation}
(recall the definition in \eqref{eq:f-alpha}).
Consider the vertex $v:=f_{\alpha_{2n}}(F(P))\circ(1)\in B_{2n}(n)\circ(1)$ and let $C$ and $C'$ denote the cycles from $\cC_{2n+1}$ or $\cC_{2n+1}'$, respectively, that contain the edge $e$ from $M_{2n+1}^{FL}$ which ends at $v$ (cf.~\eqref{eq:2-factor}). So $C$ and $C'$ share the edge $e$, the next edge on $C$ incident to $v$ is $f_{\alpha_{2n}}((F(P),S(P)))\circ(1)$, and the next edge on $C'$ incident to $v$ is given by $f_{\alpha_{2n}'}((F(P'),S(P')))\circ(1)$ for some $P'\in\cP_{2n}(n,n+1)$ with $f_{\alpha_{2n}'}(F(P'))\circ(1)=v$ (recall \eqref{eq:FL-invariant}). But by \eqref{eq:images-S-P} those edges are different in $C$ and $C'$, proving that $\cC_{2n+1}$ and $\cC_{2n+1}'$ are different subgraphs of $Q_{2n+1}(n,n+1)$.

Suppose the 2-factors $\cC_{2n+1}$ and $\cC'_{2n+1}$ are used for further construction steps by splitting them up as described in Section~\ref{sec:construction}, and consider the respective families of oriented paths defined in \eqref{eq:new-paths}. Note that both of these families contain an oriented path whose last edge is $e\circ(1)$, but whose second to last edge is different, namely $f_{\alpha_{2n}}((F(P),S(P)))\circ(1,1)$ and $f_{\alpha_{2n}'}((F(P'),S(P')))\circ(1,1)$, respectively. By the definition in \eqref{eq:ind-step2-P} this difference propagates to the path families in the layer $Q_{2n+2}(n+1,n+2)$, and hence also through all further construction steps (regardless of how the two parameter sequences are extended).
\end{proof}

The next proposition identifies pairs of parameter sequences for which the resulting 2-factors are mapped onto each other under automorphisms of $Q_{2n+1}(n,n+1)$ (in particular, the number and lengths of the cycles in each of the 2-factors are the same). Note however that even if no such automorphism exists, the number and/or the lengths of the cycles in certain 2-factors from our construction could nevertheless be the same.

\begin{proposition} \label{prop:automorphisms}
Let $n\geq 1$ and let $(\alpha_{2i})_{1\leq i\leq n}$ and $(\alpha_{2i}')_{1\leq i\leq n}$, $\alpha_{2i},\alpha_{2i}'\in\{0,1\}^{i-1}$, be two parameter sequences satisfying $\alpha_{2i}=\alpha_{2i}'$ for all $1\leq i\leq n-1$ and $\rev(\alpha_{2n})=\alpha_{2n}'$. Let $\cC_{2n+1}$ and $\cC_{2n+1}'$ denote the 2-factors defined in Section~\ref{sec:construction} for these parameter sequences, respectively.
Then the mapping $\tau_{\alpha_{2n}'}:B_{2n+1}\rightarrow B_{2n+1}$, defined by
\begin{equation} \label{eq:tau}
  (x_1,\ldots,x_{2n},x_{2n+1})\mapsto \big(f_{\alpha_{2n}'}(x_1,\ldots,x_{2n}),\ol{x_{2n+1}}\big) \enspace,
\end{equation}
where $f_{\alpha_{2n}'}$ is defined in \eqref{eq:f-alpha}, is an automorphism of $Q_{2n+1}(n,n+1)$ that maps $\cC_{2n+1}$ onto $\cC_{2n+1}'$.
\end{proposition}

For any parameter sequence $(\alpha_{2i})_{1\leq i\leq n}$ with $\rev(\alpha_{2n})=\alpha_{2n}$, Proposition~\ref{prop:automorphisms} implies that the mapping $\tau_{\alpha_{2n}}$ is an automorphism of $Q_{2n+1}(n,n+1)$ that maps the 2-factor $\cC_{2n+1}$ defined for this parameter sequence onto itself.

\begin{proof}
The mapping $\tau_{\alpha_{2n}'}$ is clearly an automorphism of $Q_{2n+1}(n,n+1)$ (this mapping just permutes and inverts the bits). It remains to show that $\tau_{\alpha_{2n}'}$ maps $\cC_{2n+1}$ onto $\cC_{2n+1}'$.

By the definition in \eqref{eq:2-factor} every cycle $C\in\cC_{2n+1}$ has the form
\begin{equation} \label{eq:cycle}
  C=\big(P^1\circ(0),f_{\alpha_{2n}}(\wh{P}^1)\circ(1),P^2\circ(0),f_{\alpha_{2n}}(\wh{P}^2)\circ(1),\ldots,P^k\circ(0),f_{\alpha_{2n}}(\wh{P}^k)\circ(1)\big) \enspace,
\end{equation}
for oriented paths $P^1,\ldots,P^k,\wh{P}^1,\ldots,\wh{P}^k\in\cP_{2n}(n,n+1)$, where for all $i=1,\ldots,k$ the vertices of each subpath $P^i\circ(0)\seq Q_{2n}(n,n+1)\circ(0)$ are visited in the order given by the orientation of $P^i$ and the vertices of each subpath $f_{\alpha_{2n}}(\wh{P}^i)\circ(1)\seq Q_{2n}(n-1,n)\circ(1)$ are visited in the order opposite to the orientation of $\wh{P}^i$.
Using that by the definition in \eqref{eq:f-alpha} and the assumption $\rev(\alpha_{2n})=\alpha_{2n}'$ the mapping $f_{\alpha_{2n}'}\circ f_{\alpha_{2n}}$ is the identity mapping, we obtain from \eqref{eq:tau} and \eqref{eq:cycle} that
\begin{equation*}
  \tau_{\alpha_{2n}'}(C)=\big(f_{\alpha_{2n}'}(P^1)\circ(1),\wh{P}^1\circ(0),f_{\alpha_{2n}'}(P^2)\circ(1),\wh{P}^2\circ(0),\ldots,f_{\alpha_{2n}'}(P^k)\circ(1),\wh{P}^k\circ(0)\big) \enspace,
\end{equation*}
which by the definition in \eqref{eq:2-factor} is a cycle in $\cC_{2n+1}'$.
\end{proof}

\begin{remark}
Computer experiments suggest that apart from the automorphisms mentioned in Proposition~\ref{prop:automorphisms}, there are no other nontrivial automorphisms of $Q_{2n+1}(n,n+1)$ that map certain 2-factors from our construction onto each other, with the following exceptions: The 2-factor $\cC_5$ in $Q_5(2,3)$ obtained for the parameter sequence $\alpha_2=()$, $\alpha_4=(1)$ is mapped onto itself under six additional automorphisms of $Q_5(2,3)$ (apart from the trivial one and the one given by Proposition~\ref{prop:automorphisms}). Furthermore, the 2-factor $\cC_{2n+1}$ in $Q_{2n+1}(n,n+1)$ obtained for the parameter sequence $(\alpha_{2i})_{1\leq i\leq n}$, $\alpha_{2i}=(0,0,\ldots,0)\in\{0,1\}^{i-1}$ is mapped onto itself under all $2(2n+1)$ automorphisms given by bit shifts and bit shifts plus reversal and inversion.
\end{remark}

\section{The number and lengths of cycles in the 2-factor}
\label{sec:structure-2-factor}

In this section we investigate the number and lengths of the cycles in the 2-factor $\cC_{2n+1}$ defined in Section~\eqref{sec:construction}. We identify a few properties that hold for \emph{any} choice of the parameter sequence $(\alpha_{2i})_{i\geq 1}$ (Theorem~\ref{thm:all-alpha-cycles} below) and then focus on \emph{one particular} parameter sequence, namely $\alpha_{2i}=(0,0,\ldots,0)\in\{0,1\}^{i-1}$ for all $i\geq 1$, for which the resulting 2-factor has several nice combinatorial properties related to plane trees (Theorem~\ref{thm:alpha0-cycles} below).

As we have seen in Section~\ref{sec:correctness}, in order to understand why our inductive construction indeed works as claimed, we only needed to consider the sets of first, second and last end vertices of the paths in the families $\cP_{2n}(k,k+1)$ (and could neglect all the other vertices on these paths). Note also that so far we did not use any knowledge about which of those vertices actually lie on the same paths. This knowledge will however be crucial in the following.

\subsection{Subpaths of lattice paths}

We begin by extending some of the notation introduced in Section~\ref{sec:lattice-paths}.

For any lattice path $p\in P_n$ and any two abscissas $0\leq x\leq x'\leq n$ we define $p[x,x']$ as the subpath of $p$ between (and including) the abscissas $x$ and $x'$.

For any lattice path $p$ in one of the sets $D_{2n}^{=0}(k)$, $D_{2n}^{>0}(k+1)$ and $D_{2n}^-(k)$, $k=n,n+1,\ldots,2n-1$, we define disjoint subpaths $\ell(p)$ and $r(p)$ of $p$ that cover all but two steps of $p$ as follows (see Figure~\ref{fig:subpaths}):
\begin{subequations} \label{eq:def-ell-r}
\begin{itemize}
\item
If $p\in D_{2n}^{=0}(k)$ we define
\begin{equation} \label{eq:def-ell-r-F}
  \ell(p):=p[1,x-1] \quad \text{and} \quad r(p):=p[x,2n] \enspace,
\end{equation}
where $x$ is the smallest strictly positive abscissa where $p$ touches the $y$-axis.

\item
If $p\in D_{2n}^{>0}(k+1)$ we define
\begin{equation} \label{eq:def-ell-r-S}
  \ell(p):=p[1,x] \quad \text{and} \quad r(p):=p[x+1,2n] \enspace,
\end{equation}
where $x$ is the largest abscissa where $p$ touches the line $y=1$ (the first time $p$ touches it is at $(1,1)$).

\item
If $p\in D_{2n}^-(k)$ we define
\begin{equation} \label{eq:def-ell-r-L}
  \ell(p):=p[0,x-1] \quad \text{and} \quad r(p):=p[x+1,2n] \enspace,
\end{equation}
where $x$ is the abscissa where $p$ touches the line $y=-1$.
\end{itemize}
\end{subequations}

With those definitions, depending on whether $p$ is contained in the set $D_{2n}^{=0}(k)$, $D_{2n}^{>0}(k+1)$ or $D_{2n}^-(k)$, we have
\begin{subequations}
\begin{align}
  p &= (\nearrow)\circ\ell(p)\circ(\searrow)\circ r(p) \enspace, \label{eq:ell-r-F-partition} \\
  p &= (\nearrow)\circ\ell(p)\circ(\nearrow)\circ r(p) \enspace, \label{eq:ell-r-S-partition} \\
  p &= \ell(p)\circ(\searrow,\nearrow)\circ r(p) \enspace, \notag %\label{eq:ell-r-L-partition}
\end{align}
\end{subequations}
respectively. In all cases, the subpath $\ell(p)$ starts and ends at the same ordinate and never moves below this ordinate in between. Furthermore, the ordinate of the endpoint of the subpath $r(p)$ is by $2(k-n)$ higher than the ordinate of its starting point and also this subpath never moves below the ordinate of its starting point.

\begin{figure}
\centering
\PSforPDF{
 \psfrag{z}{0}
 \psfrag{x}{$x$}
 \psfrag{2n}{$2n$}
 \psfrag{2k}[c][c][1][90]{\small $2(k-n)$}
 \psfrag{2kp}[c][c][1][90]{\small $2(k+1-n)$}
 \psfrag{Deq0}{$p\in D_{2n}^{=0}(k)$}
 \psfrag{Dgt0}{$p\in D_{2n}^{>0}(k+1)$}
 \psfrag{Dm}{$p\in D_{2n}^-(k)$}
 \psfrag{ellp}{$\ell(p)$}
 \psfrag{rp}{$r(p)$}
 \includegraphics{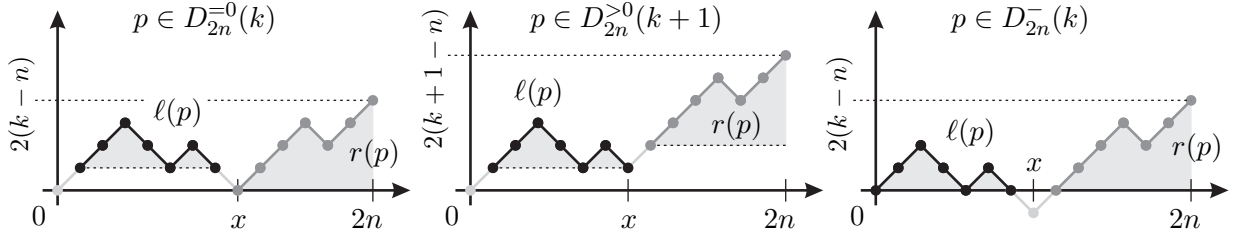}
}
\caption{Illustration of the definitions in \eqref{eq:def-ell-r}.}
\label{fig:subpaths}
\end{figure}

\subsection{Properties that are independent of the parameter sequence}

The next lemma relates the lattice paths $\varphi(F(P))$, $\varphi(S(P))$ and $\varphi(L(P))$ corresponding to the first, second and last vertex on each of the paths $P\in\cP_{2n}(k,k+1)$ arising in our construction and characterizes the length of $P$ by those lattice paths. In the following we will repeatedly use that by Lemma~\ref{lemma:FSL-D-isomorphic} those lattice paths satisfy $\varphi(F(P))\in D_{2n}^{=0}(k)$, $\varphi(S(P))\in D_{2n}^{>0}(k+1)$ and $\varphi(L(P))\in D_{2n}^-(k)$. Note that all statements of the next lemma hold independently of the parameter sequence $(\alpha_{2i})_{i\geq 1}$ chosen for the construction. In particular, the length of the paths in the families $\cP_{2n}(k,k+1)$ is independent of this parameter sequence (cf.\ the remarks in Section~\ref{sec:dependence-alpha}).

For any graph $G$ we denote by $e(G)$ the number of edges of $G$. Moreover, for any lattice path $p$ we denote by $|p|$ the number of steps of $p$.

\begin{lemma} \label{lemma:all-alpha-paths}
For any $n\geq 1$, the families of paths $\cP_{2n}(k,k+1)$, $k=n,n+1,\ldots,2n-1$, defined in Section~\ref{sec:construction} have the following properties: For any path $P\in\cP_{2n}(k,k+1)$, defining $p_F:=\varphi(F(P))\in D_{2n}^{=0}(k)$, $p_S:=\varphi(S(P))\in D_{2n}^{>0}(k+1)$ and $p_L:=\varphi(L(P))\in D_{2n}^-(k)$, we have
\begin{align}
  (\ell(p_F),r(p_F)) &= (\ell(p_S),r(p_S)) \enspace, \label{eq:FS-relation} \\
  (|\ell(p_S)|,|r(p_S)|) &= (|\ell(p_L)|,|r(p_L)|) \enspace, \label{eq:SL-length-relation} \\
  e(P) &= 2|\ell(p_F)|+2 \enspace, \label{eq:path-length}
\end{align}
where $\ell(p_F)$ and $r(p_F)$ are defined in \eqref{eq:def-ell-r-F}, $\ell(p_S)$ and $r(p_S)$ in \eqref{eq:def-ell-r-S}, and $\ell(p_L)$ and $r(p_L)$ in \eqref{eq:def-ell-r-L}.
\end{lemma}

With the equality in \eqref{eq:FS-relation} we mean that the step sequences of the lattice paths $\ell(p_F)$ and $\ell(p_S)$, and the step sequences of the lattice paths $r(p_F)$ and $r(p_S)$ are the same. The absolute coordinates of those subpaths of $p_F$ and $p_S$ might be different.

Note that by \eqref{eq:FS-relation} and \eqref{eq:SL-length-relation} the relation \eqref{eq:path-length} can also be written as $e(P)=2|\ell(p_F)|+2=2|\ell(p_S)|+2=2|\ell(p_L)|+2$.

\begin{proof}
We argue by induction over $n$. By the definition in \eqref{eq:ind-base-P}, for $n=1$ the families of paths $\cP_{2n}(k,k+1)$ consist only of a single family $\cP_2(1,2)$, which contains only a single path $P:=((1,0),(1,1),(0,1))$ ($P$ has two edges). We clearly have $p_F:=\varphi(F(P))=(\nearrow,\searrow)\in D_2^{=0}(1)$, $p_S:=\varphi(S(P))=(\nearrow,\nearrow)\in D_2^{>0}(2)$ and $p_L:=\varphi(L(P))=(\searrow,\nearrow)\in D_2^-(1)$, and by the definitions in \eqref{eq:def-ell-r} the subpaths $\ell(p_F)$, $r(p_F)$, $\ell(p_S)$, $r(p_S)$, $\ell(p_L)$ and $r(p_L)$ of those lattice paths all consist only of a single point (and zero steps), showing that all three claims of the lemma hold. This settles the induction basis.

For the induction step $n\rightarrow n+1$ let $n\geq 1$ be fixed. We consider a fixed path $P^+$ from one of the families $\cP_{2n+2}(k,k+1)$, $k=n+1,n+2,\ldots,2n+1$, and define the lattice paths $p_F^+:=\varphi(F(P^+))\in D_{2n+2}^{=0}(k)$, $p_S^+:=\varphi(S(P^+))\in D_{2n+2}^{>0}(k+1)$ and $p_L^+:=\varphi(L(P^+))\in D_{2n+2}^-(k)$.
By the definitions in \eqref{eq:ind-step1-P} and \eqref{eq:ind-step2-P}, $P^+$ is either contained in the set
\begin{equation} \label{eq:upper-paths}
  \cP_{2n}(n,n+1)\circ \{(1,0),(1,1)\} \cup \bigcup_{k'=n+1}^{2n-1} \cP_{2n}(k',k'+1)\circ\{(0,0),(1,0),(0,1),(1,1)\} 
\end{equation}
or in the set $\cP_{2n+2}'$ defined in \eqref{eq:new-paths} (in the latter case we have $k=n+1$).

We first consider the case that $P^+$ is contained in \eqref{eq:upper-paths}, i.e., $P^+$ is obtained from some path $P\in\cP_{2n}(k',k'+1)$, $n\leq k'\leq 2n-1$, by extending each vertex label of $P$ by two bits $x_1,x_2\in\{0,1\}$.
% (note that we have $k=k'+\mathbf{1}_{x_1=1}+\mathbf{1}_{x_2=1}$).
We know by induction that the lattice paths $p_F:=\varphi(F(P))\in D_{2n}^{=0}(k')$, $p_S:=\varphi(S(P))\in D_{2n}^{>0}(k'+1)$ and $p_L:=\varphi(L(P))\in D_{2n}^-(k')$ satisfy the relations
\begin{align}
  (\ell(p_F),r(p_F)) &= (\ell(p_S),r(p_S)) \enspace, \label{eq:ind-FS-relation} \\
  (|\ell(p_S)|,|r(p_S)|) &= (|\ell(p_L)|,|r(p_L)|) \enspace, \label{eq:ind-SL-length-relation} \\
  e(P) &= 2|\ell(p_F)|+2 \enspace. \label{eq:ind-path-length}
\end{align}
Moreover, we clearly have
\begin{subequations}
\begin{align}
  p_F^+ &= p_F\circ\varphi((x_1,x_2)) \enspace, \label{eq:pFp-pF} \\
  p_S^+ &= p_S\circ\varphi((x_1,x_2)) \enspace, \label{eq:pSp-pS} \\
  p_L^+ &= p_L\circ\varphi((x_1,x_2)) \enspace. \label{eq:pLp-pL}
\end{align}
\end{subequations}
Using \eqref{eq:pFp-pF} and the fact that $p_F$ is contained in the set $D_{2n}^{=0}(k')$, the definition in \eqref{eq:def-ell-r-F} yields
\begin{equation} \label{eq:ell-r-pFp-pF}
  (\ell(p_F^+),r(p_F^+)) = \big(\ell(p_F),r(p_F)\circ \varphi((x_1,x_2))\big) \enspace.
\end{equation}
Similarly, using \eqref{eq:pLp-pL} and the fact that $p_L$ is contained in the set $D_{2n}^-(k')$, the definition in \eqref{eq:def-ell-r-L} yields
\begin{equation} \label{eq:ell-r-pLp-pL}
  (\ell(p_L^+),r(p_L^+)) = \big(\ell(p_L),r(p_L)\circ \varphi((x_1,x_2))\big) \enspace.
\end{equation}
Using that $p_S\in D_{2n}^{>0}(k'+1)$, it follows that if $k'=n$, then the $y$-coordinate of the last point of $p_S$ is 2, whereas if $k'\geq n+1$, then the $y$-coordinate of the last point of $p_S$ is at least 4. Combined with \eqref{eq:upper-paths} and \eqref{eq:pSp-pS} it follows that the last two steps of $p_S^+$ do not move below the line $y=2$. By the definition in \eqref{eq:def-ell-r-S} and by \eqref{eq:pSp-pS} we therefore have
\begin{equation} \label{eq:ell-r-pSp-pS}
  (\ell(p_S^+),r(p_S^+)) = \big(\ell(p_S),r(p_S)\circ \varphi((x_1,x_2))\big) \enspace.
\end{equation}
Combining \eqref{eq:ind-FS-relation}, \eqref{eq:ell-r-pFp-pF} and \eqref{eq:ell-r-pSp-pS} yields $(\ell(p_F^+),r(p_F^+))=(\ell(p_S^+),r(p_S^+))$ and thus proves \eqref{eq:FS-relation}. Combining \eqref{eq:ind-SL-length-relation}, \eqref{eq:ell-r-pLp-pL} and \eqref{eq:ell-r-pSp-pS} yields $(|\ell(p_S^+)|,|r(p_S^+)|)=(|\ell(p_L^+)|,|r(p_L^+)|)$ and thus proves \eqref{eq:SL-length-relation}.
Using $e(P^+)=e(P)$ and \eqref{eq:ell-r-pFp-pF} we obtain from \eqref{eq:ind-path-length} that $e(P^+)=2|\ell(p_F^+)|+2$, proving \eqref{eq:path-length}.

We now consider the case that $P^+$ is contained in the set $\cP_{2n+2}'$. For the reader's convenience, Figure~\ref{fig:new-paths} illustrates the notations used in this part of the proof. By the definition in \eqref{eq:new-paths}, there are two (not necessarily distinct) paths $P,P'\in\cP_{2n}(n,n+1)$ with
\begin{subequations} \label{eq:FS-Pp-P}
\begin{align}
  F(P^+) &= S(P)\circ(0,0) \enspace, \\
  S(P^+) &= S(P)\circ(0,1) \enspace, \\
  L(P^+) &= F(P')\circ(0,1)
\end{align}
\end{subequations}
(cf.~\eqref{eq:new-paths-FSL}).
Defining $p_S:=\varphi(S(P))\in D_{2n}^{>0}(n+1)$ and $p_F':=\varphi(F(P'))\in D_{2n}^{=0}(n)$ we obtain from \eqref{eq:FS-Pp-P} that
\begin{subequations}
\begin{align}
  p_F^+ &= p_S\circ (\searrow,\searrow) \enspace, \label{eq:pFp-pS} \\
  p_S^+ &= p_S\circ (\searrow,\nearrow) \enspace, \label{eq:pSp-pS2} \\
  p_L^+ &= p_F'\circ (\searrow,\nearrow) \enspace. \label{eq:pLp-pFpr}
\end{align}
\end{subequations}
The lattice path $p_S\in D_{2n}^{>0}(n+1)$ clearly ends at $(2n,2)$.
From \eqref{eq:pFp-pS} it follows that $p_F^+\in D_{2n+2}^{=0}(n+1)$ and that the smallest strictly positive abscissa where this lattice path touches the $y$-axis is $2n+2$ (see Figure~\ref{fig:new-paths}). By the definition in \eqref{eq:def-ell-r-F} and by \eqref{eq:pFp-pS} we therefore have
\begin{equation} \label{eq:ell-r-pF+}
  (\ell(p_F^+),r(p_F^+))=\big(p_S[1,2n]\circ (\searrow),()\big)
\end{equation}
($r(p_F^+)$ consists only of a single point).
From \eqref{eq:pSp-pS2} it follows that $p_S^+\in D_{2n+2}^{>0}(n+2)$ and that the largest abscissa where this lattice path touches the line $y=1$ is $2n+1$. By the definition in \eqref{eq:def-ell-r-S} and by \eqref{eq:pSp-pS2} we therefore have
\begin{equation} \label{eq:ell-r-pS+}
  (\ell(p_S^+),r(p_S^+))=\big(p_S[1,2n]\circ (\searrow),()\big) \enspace,
\end{equation}
which together with \eqref{eq:ell-r-pF+} shows that \eqref{eq:FS-relation} also holds in this case.

From \eqref{eq:pLp-pFpr} it follows that $p_L^+\in D_{2n+2}^-(n+1)$ and that the only abscissa where this lattice path touches the line $y=-1$ is $2n+1$. By the definition in \eqref{eq:def-ell-r-L} and by \eqref{eq:pLp-pFpr} we therefore have
\begin{equation} \label{eq:ell-r-pL+}
  (\ell(p_L^+),r(p_L^+))=\big(p_F',()\big) \enspace.
\end{equation}
Together with \eqref{eq:ell-r-pS+} it follows that $(|\ell(p_S^+)|,|r(p_S^+)|)=(2n,0)=(|\ell(p_L^+)|,|r(p_L^+)|)$, proving \eqref{eq:SL-length-relation} in this case.

It remains to prove \eqref{eq:path-length} in this case.
Note that we have
\begin{equation} \label{eq:ell-pF+}
  \ell(p_F^+) \eqBy{eq:ell-r-pF+} p_S[1,2n]\circ (\searrow) \eqBy{eq:ell-r-S-partition} \ell(p_S)\circ(\nearrow)\circ r(p_S)\circ(\searrow) \enspace.
\end{equation}
By induction we have for $p_F:=\varphi(F(P))\in D_{2n}^{=0}(n)$ that
\begin{equation} \label{eq:length-pF-pS}
  (|\ell(p_F)|,|r(p_F)|) \eqBy{eq:FS-relation} (|\ell(p_S)|,|r(p_S)|)
\end{equation}
and for $p_L:=\varphi(L(P))\in D_{2n}^-(n)$ that
\begin{equation} \label{eq:length-pS-pL}
  (|\ell(p_S)|,|r(p_S)|) \eqBy{eq:SL-length-relation} (|\ell(p_L)|,|r(p_L)|) \enspace.
\end{equation}
By Lemma~\ref{lemma:FL-invariant} there is a path $\wh{P}\in\cP_{2n}(n,n+1)$ (which is not necessarily distinct from $P$ or $P'$) satisfying
\begin{equation} \label{eq:f-alpha-LolP-LP}
  f_{\alpha_{2n}}(L(\wh{P})) = L(P) \enspace.
\end{equation}
This path is relevant for us, as by the definitions in \eqref{eq:2-factor} and \eqref{eq:new-paths} we have
\begin{equation} \label{eq:length-P+}
  e(P^+)=1+(e(P)-1)+2+e(\wh{P}) \enspace,
\end{equation}
where the $+1$ counts the edge in $P^+$ that originates from the matching $M_{2n+2}^S$, the $+2$ the two edges originating from the matching $M_{2n+1}^{FL}$, and the $-1$ accounts for the fact that the edge $(F(P),S(P))$ is not contained in $P^+$ (see Figure~\ref{fig:new-paths}).
Note that the path $\wh{P}$ also satisfies
\begin{equation} \label{eq:f-alpha-FolP-FP'}
  f_{\alpha_{2n}}(F(\wh{P})) = F(P')
\end{equation}
(we do not use this relation here, though).

We define the lattice paths $\wh{p}_F:=\varphi(F(\wh{P}))\in D_{2n}^{=0}(n)$, $\wh{p}_S:=\varphi(S(\wh{P}))\in D_{2n}^{>0}(n+1)$ and $\wh{p}_L:=\varphi(L(\wh{P}))\in D_{2n}^-(n)$.

Using the second part of Lemma~\ref{lemma:dyck-paths-invariant} and the definition in \eqref{eq:def-ell-r-L} we obtain from \eqref{eq:f-alpha-LolP-LP} that
\begin{equation} \label{length-pL-olpL}
  (|r(\wh{p}_L)|,|\ell(\wh{p}_L)|)=(|\ell(p_L)|,|r(p_L)|)
\end{equation}
(recall that both $p_L$ and $\wh{p}_L$ are contained in the set $D_{2n}^-(n)$).

By induction we have
\begin{equation} \label{eq:length-olpF-olpL}
  (|\ell(\wh{p}_F)|,|r(\wh{p}_F)|) \eqBy{eq:FS-relation} (|\ell(\wh{p}_S)|,|r(\wh{p}_S)|) \eqBy{eq:SL-length-relation} (|\ell(\wh{p}_L)|,|r(\wh{p}_L)|) \enspace.
\end{equation}

Applying the induction hypothesis, we may continue \eqref{eq:length-P+} as follows:
\begin{equation} \label{eq:length-P+-contd}
\begin{split}
  e(P^+) = e(P)+e(\wh{P})+2
        &\eqBy{eq:path-length} (2\underbrace{|\ell(p_F)|}_{\eqBy{eq:length-pF-pS} |\ell(p_S)|}
                                    +2)
                             +(2\!\!\!\!\!\!\!\!\!\underbrace{|\ell(\wh{p}_F)|}_{\eqByM{\eqref{eq:length-pS-pL},\eqref{length-pL-olpL},\eqref{eq:length-olpF-olpL}} |r(p_S)| }\!\!\!\!\!\!\!\!\!
                                    +2)+2 \\
         &= 2(|\ell(p_S)|+|r(p_S)|+2)+2 \\
         &\eqBy{eq:ell-pF+} 2|\ell(p_F^+)|+2 \enspace.
\end{split}
\end{equation}
This completes the proof.
\end{proof}

The following theorem states an expression for the length of the cycles in the 2-factor $\cC_{2n+1}$ in the middle layer of $Q_{2n+1}$ arising from our construction.

\begin{theorem} \label{thm:all-alpha-cycles}
For any $n\geq 1$, the family of paths $\cP_{2n}(n,n+1)$ and the 2-factor $\cC_{2n+1}$ defined in Section~\ref{sec:construction} have the following property: For any cycle in $\cC_{2n+1}$, the distance (along the cycle) between any two neighboring vertices of the form $F(P)\circ(0)$, $F(P')\circ(0)$ with $P,P'\in \cP_{2n}(n,n+1)$ on the cycle equals $4n+2$. Consequently, for any cycle $C\in \cC_{2n+1}$ we have
\begin{equation} \label{eq:cycle-length}
  e(C)=(4n+2)\cdot|\{P\in\cP_{2n}(n,n+1)\mid F(P)\circ(0)\in C\}| \enspace.
\end{equation}
In particular, the length of all cycles in $\cC_{2n+1}$ is a multiple of $4n+2$, and the length of a shortest cycle is at least $4n+2$.
\end{theorem}

Even though Theorem~\ref{thm:all-alpha-cycles} holds for \emph{any} choice of the parameter sequence $(\alpha_{2i})_{i\geq 1}$, the cardinality of the set on the right hand side of \eqref{eq:cycle-length} does of course depend on the parameter sequence.

\begin{proof}
Fix a cycle $C$ in $\cC_{2n+1}$ and recall from the definition in \eqref{eq:2-factor} that $C$ contains at least one vertex of the form $F(P)\circ(0)$ with $P\in\cP_{2n}(n,n+1)$. We fix another path $P'\in\cP_{2n}(n,n+1)$ such that $F(P')\circ(0)$ is the closest vertex to $F(P)\circ(0)$ of this form on $C$ when walking along the cycle in the direction of the edge $(F(P),S(P))\circ(0)$ (if $C$ contains only one vertex of this form, then we set $P':=P$). By the definition in \eqref{eq:2-factor} there is a path $\wh{P}\in \cP_{2n}(n,n+1)$ (which is not necessarily distinct from $P$ or $P'$) satisfying $f_{\alpha_{2n}}(L(\wh{P}))=L(P)$ and $f_{\alpha_{2n}}(F(\wh{P}))=F(P')$ (cf.~\eqref{eq:f-alpha-LolP-LP} and \eqref{eq:f-alpha-FolP-FP'} in the proof of Lemma~\ref{lemma:all-alpha-paths}), and the distance between $F(P)\circ(0)$ and $F(P')$ along the cycle $C$ is
\begin{equation} \label{eq:distance-FP-FP'}
  e(P)+e(\wh{P})+2 \enspace,
\end{equation}
where the $+2$ counts the edges in $C$ that originate from the matching $M_{2n+1}^{FL}$ (see Figure~\ref{fig:new-paths}).
In the proof of Lemma~\ref{lemma:all-alpha-paths} we have already analyzed an expression of the form~\eqref{eq:distance-FP-FP'}. Combining \eqref{eq:ell-pF+} and \eqref{eq:length-P+-contd} shows that \eqref{eq:distance-FP-FP'} evaluates to $4n+2$, as claimed.
\end{proof}

\begin{figure}
\centering
\PSforPDF{
 \psfrag{bnp2}{$B_{2n+2}(n+2)$}
 \psfrag{bnp1}{$B_{2n+2}(n+1)$}
 \psfrag{q2n4}{$Q_{2n+1}\circ(1)$}
 \psfrag{q2n1}{$Q_{2n}\circ(0,0)$}
 \psfrag{q2n2}{$Q_{2n}\circ(0,1)$}
 \psfrag{q2n3}{$Q_{2n}\circ(1,1)$}
 \psfrag{m2}{\scriptsize $M_{2n+1}^{FL}\circ(1)$} 
 \psfrag{m1}{\scriptsize $M_{2n+2}^{S}$}
 \psfrag{pp}[c][c][1][13]{$P^+\in\cP_{2n+2}'$}
 \psfrag{pm0}{\scriptsize $P\circ(0,0)$}
 \psfrag{pm1}{\scriptsize $P\circ(0,1)$}
 \psfrag{pb}{\scriptsize $\wh{P}\circ(0,1)$}
 \psfrag{ppr}{\scriptsize $P'\circ(0,1)$}
 \psfrag{fpp}{\scriptsize $F(P^+)$}
 \psfrag{spp}{\scriptsize $S(P^+)$}
 \psfrag{lpp}{\scriptsize $L(P^+)$}
 \psfrag{fpb1}{\scriptsize $f_{\alpha_{2n}}(\wh{P})\circ(0,1)$}
 \psfrag{fpb2}{\scriptsize $f_{\alpha_{2n}}(\wh{P})\circ(1,1)$}
 \psfrag{p2n1}{\scriptsize $\cP_{2n}(n,n+1)\circ(0,0)$}
 \psfrag{p2n2}{\scriptsize $\cP_{2n}(n,n+1)\circ(0,1)$}
 \psfrag{z}{\scriptsize $0$}
 \psfrag{2n}{\scriptsize $2n$}
 \psfrag{ps}{\scriptsize $p_S=\varphi(S(P))$}
 \psfrag{pfp}{\scriptsize $p_F^+=\varphi(F(P^+))$}
 \psfrag{psp}{\scriptsize $p_S^+=\varphi(S(P^+))$}
 \psfrag{plp}{\scriptsize $p_L^+=\varphi(L(P^+))$}
 \psfrag{pfpr}{\scriptsize $p_F'=\varphi(F(P'))$}
 \psfrag{pphi}{\scriptsize $=\varphi(f_{\alpha_{2n}}(F(\wh{P})))$}
 \psfrag{pa1}{\scriptsize $p_F=\varphi(F(P))$}
 \psfrag{pa2}{\scriptsize $p_L=\varphi(L(P))$}
 \psfrag{pa4}{\scriptsize $\wh{p}_F=\varphi(F(\wh{P}))$}
 \psfrag{pa5}{\scriptsize $\wh{p}_L=\varphi(L(\wh{P}))$}
 \psfrag{f1}{\scriptsize $p_L^+=\ol{\rev}(\wh{p}_F)\circ(\searrow,\nearrow)$}
 \psfrag{f2}{\scriptsize $p_L=\ol{\rev}(\wh{p}_L)$}
 \psfrag{psi1}{\scriptsize $\psi(p_F)$}
 \psfrag{psi2}{\scriptsize $\psi(p_F')$}
 \psfrag{psiell}{\scriptsize $\psi(\ell(p_F))$}
 \psfrag{psir}{\scriptsize $\psi(r(p_F))$}
 \includegraphics{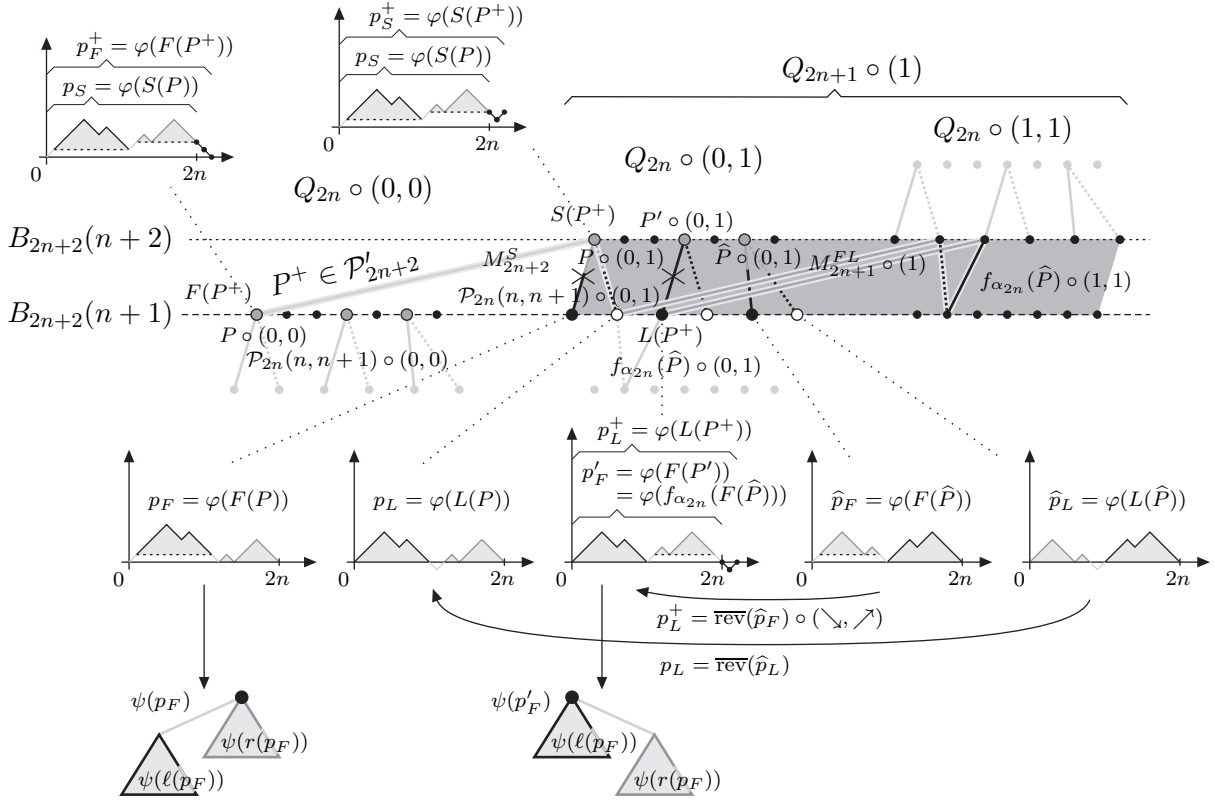}
}
\caption{Notations used in the proofs of Lemma~\ref{lemma:all-alpha-paths}, Theorem~\ref{thm:all-alpha-cycles}, Lemma~\ref{lemma:alpha0-FL-relation} and Theorem~\ref{thm:alpha0-cycles}. The figure illustrates the relations between various lattice paths corresponding to certain vertices used in our construction in the inductive step $n\rightarrow n+1$ ($Q_{2n}\rightarrow Q_{2n+2}$) when the parameter sequence $(\alpha_{2i})_{i\geq 1}$, $\alpha_{2i}=(0,0,\ldots,0)\in\{0,1\}^{i-1}$, is used. Even though for general parameter sequences certain subpaths of those lattice paths are not identical anymore, the \emph{length} of those subpaths is still the same (this is exploited in the proofs of Lemma~\ref{lemma:all-alpha-paths} and Theorem~\ref{thm:all-alpha-cycles}).}
\label{fig:new-paths}
\end{figure}

\subsection{The all-zero parameter sequence}

By the second part of Lemma~\ref{lemma:all-alpha-paths}, certain subpaths of the lattice paths $\varphi(S(P))$ and $\varphi(L(P))$ corresponding to the second and last vertex on each of the paths $P\in\cP_{2n}(k,k+1)$ arising in our construction have the same length (by Lemma~\ref{lemma:FSL-D-isomorphic} those lattice paths satisfy $\varphi(S(P))\in D_{2n}^{>0}(k+1)$ and $\varphi(L(P))\in D_{2n}^-(k)$). The following lemma states that if the all-zero parameter sequence is used for the construction, those subpaths not only have the same length, but are in fact the same (more specifically, their respective step sequences are the same).

\begin{lemma} \label{lemma:alpha0-FL-relation}
Let $n\geq 1$ and consider the parameter sequence $(\alpha_{2i})_{1\leq i\leq n-1}$ with $\alpha_{2i}=(0,0,\ldots,0)\in\{0,1\}^{i-1}$ for all $i=1,\ldots,n-1$. The families of paths $\cP_{2n}(k,k+1)$, $k=n,n+1,\ldots,2n-1$, defined in Section~\ref{sec:construction} for this parameter sequence have the following property: For any path $P\in\cP_{2n}(k,k+1)$, defining $p_S:=\varphi(S(P))\in D_{2n}^{>0}(k+1)$ and $p_L:=\varphi(L(P))\in D_{2n}^-(k)$, we have
\begin{equation} \label{eq:SL-relation}
  (\ell(p_S),r(p_S))=(\ell(p_L),r(p_L)) \enspace,
\end{equation}
where $\ell(p_S)$ and $r(p_S)$ are defined in \eqref{eq:def-ell-r-S}, and $\ell(p_L)$ and $r(p_L)$ in \eqref{eq:def-ell-r-L}.
\end{lemma}

\begin{proof}
We argue by induction over $n$. To settle the induction basis $n=1$ we argue exactly as in the proof of Lemma~\ref{lemma:all-alpha-paths}.

For the induction step $n\rightarrow n+1$ let $n\geq 1$ be fixed. We consider a fixed path $P^+$ from one of the families $\cP_{2n+2}(k,k+1)$, $k=n+1,n+2,\ldots,2n+1$, and define the lattice paths $p_S^+:=\varphi(S(P^+))\in D_{2n+2}^{>0}(k+1)$ and $p_L^+:=\varphi(L(P^+))\in D_{2n+2}^-(k)$. As argued in the proof of Lemma~\ref{lemma:all-alpha-paths}, $P^+$ is either contained in the set \eqref{eq:upper-paths} or in the set $\cP_{2n+2}'$ defined in \eqref{eq:new-paths} (in the latter case we have $k=n+1$).

The case that $P^+$ is contained in the set \eqref{eq:upper-paths} can be treated analogously as in the proof of Lemma~\ref{lemma:all-alpha-paths}: Replacing \eqref{eq:ind-SL-length-relation} by the modified induction hypothesis $(\ell(p_S),r(p_S))=(\ell(p_L),r(p_L))$ and using this relation together with \eqref{eq:ell-r-pLp-pL} and \eqref{eq:ell-r-pSp-pS} yields $(\ell(p_S^+),r(p_S^+))=(\ell(p_L^+),r(p_L^+))$ and thus proves \eqref{eq:SL-relation}. (In fact, this part of the argument does not use that $\alpha_{2n}=(0,0,\ldots,0)\in\{0,1\}^{n-1}$, but only that all other elements of the parameter sequence used in previous construction steps are zero vectors as well.)

We now focus on the more interesting case that $P^+$ is contained in the set $\cP_{2n+2}'$. For the reader's convenience, Figure~\ref{fig:new-paths} illustrates the notations used in this part of the proof. We let $P,P'\in\cP_{2n}(n,n+1)$, $p_S\in D_{2n}^{>0}(n+1)$, $p_L\in D_{2n}^-(n)$ and $p_F'\in D_{2n}^{=0}(n)$ be defined as in the proof of Lemma~\ref{lemma:all-alpha-paths}. By \eqref{eq:ell-r-pS+} and \eqref{eq:ell-r-pL+}, to complete the proof of the lemma we need to show that $p_F'=p_S[1,2n]\circ(\searrow)$.

By the definition in \eqref{eq:f-alpha}, for $\alpha_{2n}=(0,0,\ldots,0)\in\{0,1\}^{n-1}$ we have
\begin{equation} \label{eq:f-alpha-simplified}
  f_{\alpha_{2n}}=\ol{\rev} \enspace,
\end{equation}
so $f_{\alpha_{2n}}$ just reverses and inverts all bits.

By induction we have
\begin{equation} \label{eq:ell-r-pS-pL}
  (\ell(p_S),r(p_S)) = (\ell(p_L),r(p_L))
\end{equation}
(cf.~\eqref{eq:length-pS-pL}).

We let $\wh{P}\in\cP_{2n}(n,n+1)$, $\wh{p}_F\in D_{2n}^{=0}(n)$, $\wh{p}_S\in D_{2n}^{>0}(n+1)$ and $\wh{p}_L\in D_{2n}^-(n)$ be defined as in the proof of Lemma~\ref{lemma:all-alpha-paths}. Using \eqref{eq:f-alpha-simplified} the relations \eqref{eq:f-alpha-LolP-LP} and \eqref{eq:f-alpha-FolP-FP'} simplify to
\begin{equation} \label{eq:rev-LolP-LP}
  \ol{\rev}(L(\wh{P})) = L(P)
\end{equation}
and
\begin{equation} \label{eq:rev-FolP-FP'}
  \ol{\rev}(F(\wh{P})) = F(P') \enspace.
\end{equation}

Using the definition in \eqref{eq:rev-paths} we obtain from \eqref{eq:rev-LolP-LP} that
\begin{equation*}
  \ol{\rev}(\wh{p}_L)=p_L \enspace,
\end{equation*}
which by the definition in \eqref{eq:def-ell-r-L} implies that
\begin{equation} \label{eq:ell-r-olpL-pL-switch}
  \big(\ol{\rev}(r(\wh{p}_L)),\ol{\rev}(\ell(\wh{p}_L))\big) = (\ell(p_L),r(p_L))
\end{equation}
(recall that both $p_L$ and $\wh{p}_L$ are contained in the set $D_{2n}^-(n)$; cf.~\eqref{length-pL-olpL}).

By Lemma~\ref{lemma:all-alpha-paths} and by induction we have
\begin{equation} \label{eq:ell-r-olpF-olpL}
  (\ell(\wh{p}_F),r(\wh{p}_F)) \eqBy{eq:FS-relation} (\ell(\wh{p}_S),r(\wh{p}_S)) \eqBy{eq:SL-relation} (\ell(\wh{p}_L),r(\wh{p}_L))
\end{equation}
(cf.~\eqref{eq:length-olpF-olpL}).

Using the definition in \eqref{eq:rev-paths} we obtain from \eqref{eq:rev-FolP-FP'} that
\begin{equation} \label{eq:pF'}
\begin{split}
  p_F'=\ol{\rev}(\wh{p}_F)
    &\eqBy{eq:ell-r-F-partition} \ol{\rev}\big((\nearrow)\circ \ell(\wh{p}_F)) \circ(\searrow)\circ r(\wh{p}_F)\big) \\
    &\eqBy{eq:ell-r-olpF-olpL} \ol{\rev}\big((\nearrow)\circ \ell(\wh{p}_L) \circ(\searrow)\circ r(\wh{p}_L)\big) \\
    &=\ol{\rev}(r(\wh{p}_L))\circ (\nearrow)\circ \ol{\rev}(\ell(\wh{p}_L))\circ (\searrow) \\
    &\eqByM{\eqref{eq:ell-r-pS-pL},\eqref{eq:ell-r-olpL-pL-switch}} \ell(p_S) \circ (\nearrow) \circ r(p_S) \circ (\searrow) \\
    &\eqBy{eq:ell-r-S-partition} p_S[1,2n] \circ (\searrow) \enspace,
\end{split}
\end{equation}
completing the proof.
\end{proof}

In order to determine the number and lengths of the cycles in the 2-factor $\cC_{2n+1}$ for the all-zero parameter sequence, we first introduce some terminology. 

\textit{Ordered rooted trees and plane trees.}
An \emph{ordered rooted tree} is a rooted tree with a specified left-to-right ordering for the children of each vertex. We denote the set of all ordered rooted trees on $n+1$ vertices (and $n$ edges) by $\cT_{n+1}^*$. It is well known that $|\cT_{n+1}^*|=C_n$, the $n$-th Catalan number, so we have $(|\cT_{n+1}^*|)_{n\geq 1}=(1,2,5,14,42,132,429,1430,4862,16796,\ldots)$ (see \cite{catalan-seq}).

A \emph{plane tree} is a tree embedded in the plane. We denote the set of all plane trees on $n+1$ vertices by $\cT_{n+1}$. The number of plane trees is given by
\begin{subequations} \label{eq:plane-trees}
\begin{equation}
  |\cT_{n+1}| = r_{n+1}-\frac{1}{2} \big(C_n-C_{\frac{n-1}{2}} \cdot \mathbf{1}_{n\in 2\mathbb{Z}+1}\big) \enspace,
\end{equation}
where $\mathbf{1}_{n\in 2\mathbb{Z}+1}\in\{0,1\}$ denotes the indicator function for $n$ being odd and
\begin{equation}
  r_{n+1} := \frac{1}{2n} \sum_{d|n} \phi(n/d) \binom{2d}{d}
\end{equation}
\end{subequations}
with the Euler totient function $\phi$.
% (the term $r_{n+1}$ counts the number of so-called \emph{rooted plane trees} on $n+1$ vertices which can be seen as rooted ordered trees where the order of the children of the root is determined only up to circular rotation).
We have $(|\cT_{n+1}|)_{n\geq 1}=(1,1,2,3,6,14,34,95,280,854,\ldots)$ and
\begin{equation*}
  |\cT_{n+1}|=(1+o(1))\frac{4^n}{2\sqrt{\pi}n^{5/2}}
\end{equation*}
(see \cite{plane-tree-seq}).

\textit{Rotation of ordered rooted trees.}
We say that a tree $T'\in \cT_{n+1}^*$ is obtained by a \emph{rotation operation} from a tree $T\in \cT_{n+1}^*$ with root vertex $r$ and children $v_1,\ldots,v_l$ ($v_1$ is the leftmost child and $v_l$ the rightmost child), if $T'$ is obtained from $T$ by making $v_1$ the new root vertex, and the subtree rooted at $r$ without $v_1$ and its descendants (this subtree only contains $r$ and the subtrees rooted at $v_2,\ldots,v_l$) a new rightmost child of $v_1$. Intuitively, this operation rotates a tree to the right by shifting the root vertex to the left.

\newcommand{\rot}{\curvearrowright}

For two trees $T,T'\in \cT_{n+1}^*$ we write $T\rot T'$, if $T'$ can be obtained from $T$ by a sequence of rotation operations. Note that $\rot$ defines an equivalence relation on the set $\cT_{n+1}^*\times \cT_{n+1}^*$. For any $T\in \cT_{n+1}^*$ we denote by $[T]_{\rot}:=\{T'\in \cT_{n+1}^*\mid T\rot T'\}$ the corresponding equivalence class.
%, and we define the quotient set $\cT_{n+1}^*/\!\rot{}\!:=\{[T]_{\rot}\mid T\in \cT_{n+1}^*\}$.
Observe that two ordered rooted trees are equivalent in this sense, if and only if they represent the same plane tree when we embed them into the plane and unmark the root vertices (where for an ordered rooted tree we first embed the root vertex and then recursively all children from left to right according to the specified ordering).
% We therefore have $|\cT_{n+1}^*/\!\rot{}\!|=|\cT_{n+1}|$.

We say that a plane tree $T\in \cT_{n+1}$ is \emph{asymmetric} if $|[T^*]_{\rot}|=2n$, where $T^*\in \cT_{n+1}^*$ is obtained by rooting $T$ arbitrarily. Equivalently, $T$ is asymmetric if there is no nontrivial rotation operation of the plane that maps $T$ onto itself. We denote the set of all asymmetric plane trees on $n+1$ vertices by $\cT_{n+1}^{\mathrm{asym}}$. The number of asymmetric plane trees is given by
\begin{equation*}
  |\cT_{n+1}^{\mathrm{asym}}| = \rhat_{n+1}-\frac{1}{2} \big(C_n+C_{\frac{n-1}{2}} \cdot \mathbf{1}_{n\in 2\mathbb{Z}+1}\big) \enspace,
\end{equation*}
where
\begin{equation*}
  \rhat_{n+1} := \frac{1}{2n} \sum_{d|n} \mu(n/d) \binom{2d}{d}
\end{equation*}
with the Möbius function $\mu$.
We have $(|\cT_{n+1}^{\mathrm{asym}}|)_{n\geq 1}=(0,0,0,1,3,9,28,85,262,827,\ldots)$ and
\begin{equation} \label{eq:number-asymmetric-trees}
  |\cT_{n+1}^{\mathrm{asym}}|=(1-o(1))|\cT_{n+1}|
\end{equation}
(see \cite{asymmetric-plane-tree-seq}).

\textit{Bijection $\psi$ between lattice paths and ordered rooted trees.}
We define an \emph{ordered rooted tree with an active vertex} as a pair $(T,v)$, where $T$ is an ordered rooted tree and $v$ is a vertex of $T$. If $v$ equals the root of $T$, then $(T,v)$ can be identified with the ordinary ordered rooted tree $T$.

We inductively define a mapping $\psi$ that assigns to any lattice path in one of the sets $D_{n}(k)$, $n\geq 0$, $0\leq k\leq n$, defined in Section~\ref{sec:lattice-paths} an ordered rooted tree with an active vertex, as follows:
If $n=0$, then $D_0(0)$ contains only the lattice path $p$ that consists of the single point $(0,0)$. For this $p$ we define $\psi(p)$ to be the ordered root tree that consists only of a single vertex, and we define the active vertex to be the root vertex. If $n\geq 1$, then for any $0\leq k\leq n$ and any lattice path $p\in D_{n}(k)$, $p=(p_1,\ldots,p_{n-1},p_n)$, we define $p^-:=(p_1,\ldots,p_{n-1})$ and consider the tree $\psi(p^-)=:(T,v)$ ($v$ is the active vertex of this tree). We distinguish the cases whether the last step of the path $p$ is an upstep, $p_n=\nearrow$, or a downstep, $p_n=\searrow$. If $p_n=\nearrow$, then we define $T+w$ as the tree that is obtained from $T$ by adding a new vertex $w$ as the rightmost child of $v$, and define $\psi(p):=(T+w,w)$. If $p_n=\searrow$, then we define $\psi(p):=(T,u)$, where $u$ is the parent vertex of $v$.

Note that for any $p\in D_n(k)$, the tree $\psi(p)$ has $k$ edges, $k+1$ vertices and the active vertex is at depth $2k-n$ in the rightmost branch. It follows that the mapping $\psi|_{D_n(k)}$ is a bijection between $D_n(k)$ and all ordered rooted trees with $k+1$ vertices and an active vertex at depth $2k-n$ in the rightmost branch. In particular, $\psi|_{D_{2n}^{=0}(n)}$ is a bijection between $D_{2n}^{=0}(n)$ and $\cT_{n+1}^*$.

\begin{lemma} \label{lemma:rotation}
Let $n\geq 1$ and consider the parameter sequence $(\alpha_{2i})_{1\leq i\leq n}$ with $\alpha_{2i}=(0,0,\ldots,0)\in\{0,1\}^{i-1}$ for all $i=1,\ldots,n$. The family of paths $\cP_{2n}(n,n+1)$ and the 2-factor $\cC_{2n+1}$ defined in Section~\ref{sec:construction} for this parameter sequence have the following property:
For any cycle in $\cC_{2n+1}$ and any two neighboring vertices of the form $F(P)\circ(0)$, $F(P')\circ(0)$ with $P,P'\in\cP_{2n}(n,n+1)$ on the cycle, we have for $p_F:=\varphi(F(P))\in D_{2n}^{=0}(n)$ and $p_F':=\varphi(F(P'))\in D_{2n}^{=0}(n)$ that the corresponding ordered rooted trees $\psi(p_F)$ and $\psi(p_F')$ from the set $\cT_{n+1}^*$ differ by exactly one rotation operation.
\end{lemma}

\begin{proof}
Fix a cycle $C$ in $\cC_{2n+1}$ and recall from the definition in \eqref{eq:2-factor} that $C$ contains at least one vertex of the form $F(P)\circ(0)$ with $P\in\cP_{2n}(n,n+1)$. We fix another path $P'\in\cP_{2n}(n,n+1)$ such that $F(P')\circ(0)$ is the closest vertex to $F(P)\circ(0)$ of this form on $C$ when walking along the cycle in the direction of the edge $(F(P),S(P))\circ(0)$ (if $C$ contains only one vertex of this form, then we set $P':=P$). By the definition in \eqref{eq:2-factor} there is a path $\wh{P}\in \cP_{2n}(n,n+1)$ (which is not necessarily distinct from $P$ or $P'$) satisfying $f_{\alpha_{2n}}(L(\wh{P}))=L(P)$ and $f_{\alpha_{2n}}(F(\wh{P}))=F(P')$ (cf.~\eqref{eq:f-alpha-LolP-LP} and \eqref{eq:f-alpha-FolP-FP'} in the proof of Lemma~\ref{lemma:all-alpha-paths}). As in the proof of Lemma~\ref{lemma:alpha0-FL-relation}, for $\alpha_{2n}=(0,0,\ldots,0)\in\{0,1\}^{n-1}$ those relations can be simplified to show that the lattice paths $p_S:=\varphi(S(P))\in D_{2n}^{>0}(n+1)$ and $p_F':=\varphi(F(P'))\in D_{2n}^{=0}(n)$ satisfy
\begin{equation} \label{eq:pF'-pS}
  p_F'=\ell(p_S) \circ (\nearrow) \circ r(p_S) \circ (\searrow)
\end{equation}
(cf.~\eqref{eq:pF'}).
Defining $p_F:=\varphi(F(P))\in D_{2n}^{=0}(n)$ and applying the first part of Lemma~\ref{lemma:all-alpha-paths} shows that \eqref{eq:pF'-pS} can be written as
\begin{equation} \label{eq:pF'-pF}
  p_F'=\ell(p_F) \circ (\nearrow) \circ r(p_F) \circ (\searrow) \enspace. 
\end{equation}
We also know that
\begin{equation} \label{eq:pF-partition}
  p_F \eqBy{eq:ell-r-F-partition} (\nearrow)\circ \ell(p_F) \circ (\searrow) \circ r(p_F) \enspace.
\end{equation}
Note that by the definition in \eqref{eq:def-ell-r-F}, the subpaths $\ell(p_F)$ and $r(p_F)$ of $p_F$ start and end at the ordinate $y=1$ or $y=0$, respectively, and never move below this ordinate in between. It follows that for both ordered rooted trees $\psi(\ell(p_F))$ and $\psi(r(p_F))$ the active vertex equals the root vertex. Using this observation and the relations \eqref{eq:pF'-pF} and \eqref{eq:pF-partition} shows that $\psi(p_F')$ can be obtained from $\psi(p_F)$ by one rotation operation (see the bottom part of Figure~\ref{fig:new-paths}), as claimed.
\end{proof}

The next theorem shows that for the all-zero parameter sequence, the cycles in the 2-factor $\cC_{2n+1}$ are intimately related to the set $\cT_{n+1}^*$ of ordered rooted trees under the equivalence relation $\rot$. We thus obtain very precise information about the number and lengths of those cycles.

\begin{theorem} \label{thm:alpha0-cycles}
Let $n\geq 1$ and consider the parameter sequence $(\alpha_{2i})_{1\leq i\leq n}$ with $\alpha_{2i}=(0,0,\ldots,0)\in\{0,1\}^{i-1}$ for all $i=1,\ldots,n$. The 2-factor $\cC_{2n+1}$ defined in Section~\ref{sec:construction} for this parameter sequence has the following property:
There is a bijection between the cycles in $\cC_{2n+1}$ and the trees in the set $\cT_{n+1}$ such that any cycle $C\in\cC_{2n+1}$ and any tree $T\in \cT_{n+1}$ that are mapped onto each other satisfy the relation
\begin{equation} \label{eq:alpha0-cycle-length}
  e(C) = (4n+2)\cdot |[T^*]_{\rot}| \enspace,
\end{equation}
where $T^*\in \cT_{n+1}^*$ is obtained by rooting $T$ arbitrarily.

Consequently, the length of a shortest cycle in $\cC_{2n+1}$ is $2(4n+2)$ for all $n\geq 2$ and the length of a longest cycle is $2n(4n+2)$ for all $n\geq 4$.
Furthermore, the total number of cycles in the 2-factor is $|\cC_{2n+1}|=|\cT_{n+1}|$, and the number of cycles of length $2n(4n+2)$ is $|\cT_{n+1}^{\mathrm{asym}}|=(1-o(1))|\cT_{n+1}|$.
\end{theorem}

\begin{proof}
Fix a cycle $C\in\cC_{2n+1}$ and consider the set of paths $\cP(C):=\{P\in\cP_{2n}(n,n+1)\mid F(P)\circ(0)\in C\}$. By Lemma~\ref{lemma:rotation} the corresponding set $\cT^*(C):=\{\psi(\varphi(F(P)))\mid P\in \cP(C)\}$ forms an equivalence class of ordered rooted trees from the set $\cT_{n+1}^*$ under the rotation operation $\rot$, i.e., when embedding the trees from $\cT^*(C)$ into the plane and unmarking the root vertices these trees all represent the same plane tree $T\in\cT_{n+1}$. Put differently, we have $\cT^*(C)=[T^*]_{\rot}$, where $T^*\in\cT_{n+1}^*$ is obtained by rooting $T$ arbitrarily. We define the desired mapping by assigning to the cycle $C$ the plane tree $T$.

Using Theorem~\ref{thm:all-alpha-cycles} it follows that $e(C)=(4n+2)\cdot |\cP(C)|=(4n+2)\cdot |\cT^*(C)|=(4n+2)\cdot |[T^*]_{\rot}|$, proving the first part of the theorem.

To conclude that the above mapping between the cycles in $\cC_{2n+1}$ and the trees in the set $\cT_{n+1}$ is indeed a bijection it remains to show that \emph{all} plane trees from $\cT_{n+1}$ indeed appear as images: To see this, observe that by the definition in \eqref{eq:2-factor} for \emph{every} path $P\in\cP_{2n}(n,n+1)$, the vertex $F(P)\circ(0)$ is contained in some cycle in $\cC_{2n+1}$. As by Lemma~\ref{lemma:FSL-D-isomorphic} we have $\varphi(F(P_{2n}(n,n+1)))=D_{2n}^{=0}(n)$, it follows that for \emph{every} ordered rooted tree $T^*\in \cT_{n+1}^*$, there is a cycle $C\in\cC_{2n+1}$ such that $\cT^*(C)=[T^*]_{\rot}$.

The claims about the length of a shortest and a longest cycle in $\cC_{2n+1}$ follow immediately from this one-to-one correspondence and from \eqref{eq:alpha0-cycle-length} by observing that the smallest equivalence class $[T^*]_{\rot}$ for some $T^*\in \cT_{n+1}^*$ has exactly 2 elements for all $n\geq 2$ (for $T^*$ being a star with $n$ rays), and the largest equivalence class has exactly $2n$ elements for all $n\geq 4$ (e.g.\ for $T^*$ being the graph obtained from a star with 3 rays by extending one of the rays by a path on $n-3$ edges).

The claims about the total number of cycles in the 2-factor and the number of cycles of length $2n(4n+2)$ also follow from this one-to-one-correspondence and from \eqref{eq:number-asymmetric-trees}.
\end{proof}

\section{Computer experiments}
\label{sec:experiments}

With the help of a computer we systematically explored the effect of the parameter sequence $(\alpha_{2i})_{1\leq i\leq n}$, $\alpha_{2i}\in\{0,1\}^{i-1}$, on the number and lengths of the cycles in the 2-factor $\cC_{2n+1}$ defined in Section~\ref{sec:construction}. Our focus here is primarily on finding parameters for which $\cC_{2n+1}$ consists of a single cycle, which is a Hamiltonian cycle, or of two cycles, which can always be connected to form a Hamiltonian path in the middle layer graph $Q_{2n+1}(n,n+1)$.
As there are in total $\prod_{i=1}^n 2^{i-1}=2^{\binom{n}{2}}$ parameter sequences, searching the entire parameter space quickly becomes infeasible. Consequently, we searched the entire parameter space only for every $n\leq 7$, and for every $8\leq n\leq 14$ we searched a small fraction of it until we found 100 parameter sequences for which the 2-factor $\cC_{2n+1}$ yields a Hamiltonian cycle or path. Those experimental results are summarized in Table~\ref{tab:experiments}.

\begin{table}
\centering
\begin{tabular}{|l|l|l|l|} \hline
$n$ & \# of sequences & \# of sequences \\
    & with $|\cC_{2n+1}|=1$ & with $|\cC_{2n+1}|=2$ \\  \hline\hline
1   & 1 & 0 \\
2   & 1 & 1 \\
3   & 2 & 3 \\
4   & 6 & 12 \\
5   & 44 & 100 \\
6   & 614 & 1580 \\
7   & 0 & 113438 \\
8   & $\geq 100$ & $\geq 100$ \\
9   & $\geq 100$ & $\geq 100$ \\
10  & $\geq 100$ & $\geq 100$ \\
11  & 0          & $\geq 100$ \\
12  & $\geq 100$ & $\geq 100$ \\
13  & 0          & $\geq 100$ \\
14  & 0          & $\geq 100$  \\ \hline
% 15  & 0          & $\geq 10$  \\ \hline
\end{tabular}
\caption{Number of parameter sequences $(\alpha_{2i})_{1\leq i\leq n}$, $\alpha_{2i}\in\{0,1\}^{i-1}$, for which the 2-factor $\cC_{2n+1}$ defined in Section~\ref{sec:construction} yields a Hamiltonian cycle ($|\cC_{2n+1}|=1$) or a Hamiltonian path ($|\cC_{2n+1}|=2$) in the middle layer graph $Q_{2n+1}(n,n+1)$.}
\label{tab:experiments}
\end{table}

As the table shows, our construction indeed yields many Hamiltonian paths and cycles in $Q_{2n+1}(n,n+1)$ for $n\leq 14$.
However, as we can see from the second column, for $n\in\{7,11,13,14\}$ we did not find any 2-factor $\cC_{2n+1}$ consisting of a single cycle. The next theorem explains this phenomenon by stating an explicit expression for the parity of the number of cycles in $\cC_{2n+1}$ for all $n\geq 1$ (see \cite{heindl} for a proof).

To state the result we need the following definition:
For any $n\geq 1$ define $\beta_n=(\beta_n(1),\ldots,\beta_n(n-1))\in\{0,1\}^{n-1}$ by setting for all $i=1,\ldots,n-1$
\begin{equation} \label{eq:def-beta}
  \beta_n(i):=\begin{cases}
                1 & \text{if $\{i,n-i\}\seq \{2^k\mid k\geq 0\}$} \enspace, \\
                0 & \text{otherwise} \enspace.
              \end{cases}
\end{equation}
Note that $\beta_n$ is symmetric, that it contains at most 2 entries equal to 1, and that it is the zero vector for all $n$ that are not a sum of two powers of 2 (those are $n\in\{7,11,13,14,15,19,\ldots\}$).

\begin{theorem} \label{thm:parity}
Let $n\geq 1$ and $(\alpha_{2i})_{1\leq i\leq n}$, $\alpha_{2i}\in\{0,1\}^{i-1}$, an arbitrary parameter sequence. The number of cycles in the 2-factor $\cC_{2n+1}$ defined in Section~\ref{sec:construction} for this parameter sequence satisfies
\begin{equation} \label{eq:parity}
  |\cC_{2n+1}|\equiv \alpha_{2n}\cdot\beta_n+\mathbf{1}_{n\in \{2^k\mid k\geq 0\}}\pmod 2 \enspace,
\end{equation}
where $\beta_n$ is defined in \eqref{eq:def-beta}, $\alpha_{2n}\cdot\beta_n$ denotes the scalar product of the vectors $\alpha_{2n}$ and $\beta_n$, and $\mathbf{1}_{n\in \{2^k\mid k\geq 0\}}\in\{0,1\}$ the indicator function for $n$ being a power of 2.
\end{theorem}

By Theorem~\ref{thm:parity} the parity of $|\cC_{2n+1}|$ is controlled by only very few bits from the parameter vector $\alpha_{2n}$ (which is used in the last step of the construction of $\cC_{2n+1}$, cf.\ \eqref{eq:f-alpha} and \eqref{eq:2-factor}). In particular, if $\beta_n$ is the zero vector, the number of cycles in $\cC_{2n+1}$ is even regardless of the choice of the parameter sequence, which explains the zeros in the second column of Table~\ref{tab:experiments}.

The term $\mathbf{1}_{n\in \{2^k\mid k\geq 0\}}$ in \eqref{eq:parity} originates from the parity of the number of plane trees on $n+1$ vertices $|\cT_{n+1}|$, as for the all-zero parameter sequence we have $|\cC_{2n+1}|=|\cT_{n+1}|$ (recall Theorem~\ref{thm:alpha0-cycles} and \eqref{eq:plane-trees}).

\bibliographystyle{alpha}
\bibliography{refs}

\begin{thebibliography}{OEI11b}

\bibitem[Aig73]{MR0319772}
M.~Aigner.
\newblock Lexicographic matching in {B}oolean algebras.
\newblock {\em J. Combin. Theory Ser. B}, 14:187--194, 1973.

\bibitem[BW84]{MR737262}
M.~Buck and D.~Wiedemann.
\newblock Gray codes with restricted density.
\newblock {\em Discrete Math.}, 48(2-3):163--171, 1984.

\bibitem[DKS94]{MR1268348}
D.~Duffus, H.~Kierstead, and H.~Snevily.
\newblock An explicit {$1$}-factorization in the middle of the {B}oolean
  lattice.
\newblock {\em J. Combin. Theory Ser. A}, 65(2):334--342, 1994.

\bibitem[DSW88]{MR962223}
D.~Duffus, B.~Sands, and R.~Woodrow.
\newblock Lexicographic matchings cannot form {H}amiltonian cycles.
\newblock {\em Order}, 5(2):149--161, 1988.

\bibitem[GK76]{MR0389608}
C.~Greene and D.~Kleitman.
\newblock Strong versions of {S}perner's theorem.
\newblock {\em J. Combin. Theory Ser. A}, 20(1):80--88, 1976.

\bibitem[Hav83]{MR737021}
I.~Havel.
\newblock Semipaths in directed cubes.
\newblock In {\em Graphs and other combinatorial topics ({P}rague, 1982)},
  volume~59 of {\em Teubner-Texte Math.}, pages 101--108. Teubner, Leipzig,
  1983.

\bibitem[Hei11]{heindl}
A.~Heindl.
\newblock Properties of certain 2-factors in the middle layer graph.
\newblock Bachelor thesis, ETH Zürich, Switzerland, 2011.

\bibitem[Joh04]{MR2046083}
J.~Johnson.
\newblock Long cycles in the middle two layers of the discrete cube.
\newblock {\em J. Combin. Theory Ser. A}, 105(2):255--271, 2004.

\bibitem[KT88]{MR962224}
H.~Kierstead and W.~Trotter.
\newblock Explicit matchings in the middle levels of the {B}oolean lattice.
\newblock {\em Order}, 5(2):163--171, 1988.

\bibitem[Lov70]{MR0263646}
L.~Lov{\'a}sz.
\newblock {\em Problem 11, in {C}ombinatorial structures and their
  applications}, volume 1969 of {\em Proceedings of the Calgary International
  Conference on Combinatorial Structures and their Applications held at the
  University of Calgary, Calgary, Alberta, Canada, June}, pages xvi+508.
\newblock Gordon and Breach Science Publishers, New York, 1970.

\bibitem[OEI11a]{catalan-seq}
{The On-Line Encyclopedia of Integer Sequences, Sequence A000108}.
\newblock \url{http://oeis.org}, 2011.

\bibitem[OEI11b]{plane-tree-seq}
{The On-Line Encyclopedia of Integer Sequences, Sequence A002995}.
\newblock \url{http://oeis.org}, 2011.

\bibitem[OEI11c]{asymmetric-plane-tree-seq}
{The On-Line Encyclopedia of Integer Sequences, Sequence A005354}.
\newblock \url{http://oeis.org}, 2011.

\bibitem[SA11]{shimada-amano}
M.~Shimada and K.~Amano.
\newblock A note on the middle levels conjecture.
\newblock {\it arXiv:0912.4564}, September 2011.

\bibitem[SSS09]{MR2548541}
I.~Shields, B.~Shields, and C.~Savage.
\newblock An update on the middle levels problem.
\newblock {\em Discrete Math.}, 309(17):5271--5277, 2009.

\end{thebibliography}

\end{document}